\documentclass[12pt,a4paper,reqno]{amsart}

\usepackage{amsfonts}
\usepackage{amsmath}
\usepackage{amssymb}
\usepackage{amsthm}
\usepackage{amsxtra}
\usepackage{bbm}
\usepackage{braket}
\usepackage{comment}
\usepackage{extarrows}
\usepackage{graphicx}
\usepackage{latexsym}
\usepackage{mathrsfs}
\usepackage{yhmath}
\usepackage{nicematrix}
\usepackage{mathtools}

\usepackage{float}
\usepackage{booktabs}
\usepackage{verbatim}
\usepackage{xcolor}

\usepackage[pdfborder={0 0 0}]{hyperref} 
\hypersetup{breaklinks, raiselinks}

\usepackage{bm}
\usepackage{tikz}
\usepackage[initials,lite]{amsrefs} 
\AtBeginDocument{%
    \def\MR#1{}
}

\usepackage{cleveref}

\usepackage{a4wide}
\usepackage{fullpage}

\theoremstyle{plain}
\newtheorem{Theorem}{Theorem}[section]
\newtheorem{Lemma}[Theorem]{Lemma}
\newtheorem{Corollary}[Theorem]{Corollary}
\newtheorem{Proposition}[Theorem]{Proposition}

\theoremstyle{definition}

\newtheorem{Assumptions and Discussion}[Theorem]{Assumptions and Discussion}

\newtheorem{Definition}[Theorem]{Definition}

\newtheorem{Remark}[Theorem]{Remark}

\newtheorem{Observation}[Theorem]{Observation}
\newtheorem{Observations}[Theorem]{Observations}
\newtheorem{Notation}[Theorem]{Notation}

\theoremstyle{remark}

\newtheorem*{acknowledgment*}{Acknowledgment}

\usepackage[shortlabels]{enumitem}
\SetEnumerateShortLabel{a}{\textup{(\alph*)}}
\SetEnumerateShortLabel{A}{\textup{(\Alph*)}}
\SetEnumerateShortLabel{1}{\textup{(\arabic*)}}
\SetEnumerateShortLabel{i}{\textup{(\roman*)}}
\SetEnumerateShortLabel{I}{\textup{(\Roman*)}}

\def\lex{\operatorname{lex}}

\def\ceil#1{\left\lceil #1 \right\rceil}
\def\Char{\operatorname{char}}

\def\deg{\operatorname{deg}}

\def\dim{\operatorname{dim}}

\def\dividesnot{{\not |\,\,}}

\def\Ess{\operatorname{Ess}}

\def\floor#1{\left\lfloor #1 \right\rfloor}

\def\ini{\operatorname{in}} 

\def\KK{{\mathbb K}}

\def\lex{{\operatorname{lex}}}

\def\Mon{\operatorname{Mon}} 

\def\NN{{\mathbb N}}

\def\onto{\twoheadrightarrow}

\def\revlex{\operatorname{revlex}}

\def\sort{\operatorname{sort}}

\def\ZZ{{\mathbb Z}}

\newcommand\bdalpha{{\bm \alpha}}
\newcommand\bdA{{\bm A}}
\newcommand\bda{{\bm a}}
\newcommand\bdbeta{{\bm \beta}}

\newcommand\bdb{{\bm b}}

\newcommand\bdgamma{{\bm \gamma}}

\newcommand\bdH{{\bm H}}

\newcommand\bdlambda{{\bm \lambda}}
\newcommand\bdl{{\bm \ell}}

\newcommand\bdx{{\bm x}}
\newcommand\bdY{{\bm Y}}

\newcommand\bfF{\mathbf{F}}

\newcommand\bfG{\mathbf{G}}

\newcommand\bfX{\mathbf{X}}

\newcommand\calF{\mathcal{F}}
\newcommand\calG{\mathcal{G}}

\newcommand\calI{\mathcal{I}}

\newcommand\calR{\mathcal{R}}

\newcommand\Supp{\operatorname{Supp}}

\newcommand{\pd}{\operatorname{pd}}

\newcommand{\projdim}{\operatorname{pd}}

\def\r{\operatorname{r}}
\def\reg{\operatorname{reg}}

\begin{document}

\title{Blow-up algebras of secant varieties of rational normal scrolls}

\author[Kuei-Nuan Lin, Yi-Huang Shen]{Kuei-Nuan Lin and Yi-Huang Shen}




\address{Department of Mathematics, The Penn State University, 
McKeesport, PA, 15132, USA}
\email{kul20@psu.edu}

\address{CAS Wu Wen-Tsun Key Laboratory of Mathematics, School of Mathematical Sciences, University of Science and Technology of China, Hefei, Anhui, 230026, P.R.~China}
\email{yhshen@ustc.edu.cn}

\begin{abstract}
    In this paper, we are mainly concerned with the blow-up algebras of the secant varieties of balanced rational normal scrolls. In the first part, we give implicit defining equations of their associated Rees algebras and fiber cones. Consequently, we can tell that the fiber cones are Cohen--Macaulay normal domains. Meanwhile, these fiber cones have rational singularities in characteristic zero, and are $F$-rational in positive characteristic. The Gorensteinness of the fiber cones can also be characterized. In the second part, we compute the Castelnuovo--Mumford regularities and $\bda$-invariants of the fiber cones. We also present the reduction numbers of the ideals defined by the secant varieties.
\end{abstract}

\maketitle 

\section{Introduction}
Let $R$ be a standard graded polynomial ring over some field $\KK$ and $I$ be an ideal of $R$ minimally generated by some forms $f_1,\ldots, f_s$ of the same degree. Those forms define a rational map $\phi$ whose image is a variety $X$. The bi-homogeneous coordinate ring of the graph of $\phi$ is the Rees algebra of the ideal $I$, and the homogeneous coordinate ring of the image is the fiber cone (special fiber ring) of the ideal $I$. It is a classical problem to find the implicit defining equations of the Rees algebra and thereby of the variety $X$; see, for instance, \cite{MR3864202},  \cite{KPU2}, \cite{KPU},  \cite{Lin-Shen2} and \cite{Lin-Shen}. This problem is known for its importance and difficulty in the elimination theory. It also appears naturally in the applied mathematics fields such as geometry modelings (in the form of the moving curve theory, \cite{CD} and \cite{Cox}) and chemical reaction networks (\cite{Cox-Lin-Sosa}). 

Here, we are mostly interested in the case when the ideal $I$ describes the secant variety of a rational normal scroll.
Rational normal scrolls and their secant varieties are typical determinantal varieties, 
central in the study of algebraic varieties.
The study of determinantal
varieties has attracted earnest attentions of algebraic geometers and  commutative algebraists, partly due to the beautiful structures involved and the interesting applications to the applied mathematics and statistics; see, for instance, \cite{BV}, \cite{MR3836659}, \cite{arXiv:2005.02909} and \cite{arXiv:2003.14232}, to name but a few.

It is well-known (\cite{MR1416564}) that the rational normal scroll is uniquely determined by some sequence of positive integers $n_1,\dots,n_d$, up to projective equivalence. And in suitable coordinates, the ideal of this rational normal scroll can be generated by the maximal minors of the matrix
\begin{equation} 
    \bfX  \coloneqq  
    \scalebox{0.9}{$
        \begin{pNiceArray}{cccc|cccc|c|cccc} 
            x_{1,0} & x_{1,1} & \cdots & x_{1,n_1-1} &
            x_{2,0} & x_{2,1} & \cdots & x_{2,n_2-1} & \cdots &
            x_{d,0} & x_{d,1} & \cdots & x_{d,n_d-1} \\
            x_{1,1} & x_{1,2} & \cdots & x_{1,n_1} &
            x_{2,1} & x_{2,2} & \cdots & x_{2,n_2} & \cdots &
            x_{d,1} & x_{d,2} & \cdots & x_{d,n_d}
        \end{pNiceArray}.$} 
        \label{eqn:matrix-X}
\end{equation} 

The implicitization problem for the blow-up algebras of the rational normal scrolls is very difficult. It was until recently solved by Sammartano in full generality in \cite{MR4068250}. With this important progress, he was then able to prove that the Rees algebra and the fiber cones of the rational normal scrolls are Koszul algebras. After that, the authors of the current paper proved in \cite{Lin-Shen3} the Cohen--Macaulayness of the fiber cones of the rational normal scroll.

However, the special case when $|n_{i}-n_{j}|\le 1$ for all $i,j$ has already been studied by Conca, Herzog, and Valla \cite{Sagbi} in 1996. This case is called \emph{balanced} nowadays. Under this assumption, the matrix $\bfX$ in equation \eqref{eqn:matrix-X} can be rewritten as a special case of the extended Hankel matrix. Recall that the $r\times c$ \emph{extended Hankel matrix} is the matrix
\begin{equation}
    \bdH_{r,c,d}\coloneqq
    \begin{pmatrix}
        x_1          & x_2          & x_3          & \cdots & x_c          \\
        x_{1+d}      & x_{2+d}      & x_{3+d}      & \cdots & x_{c+d}      \\
        \vdots       & \vdots       & \vdots       &        & \vdots       \\
        x_{1+(r-1)d} & x_{2+(r-1)d} & x_{3+(r-1)d} & \cdots & x_{c+(r-1)d}
    \end{pmatrix},
    \label{HMatrix}
\end{equation}
usually considered over the ring $R=\mathbb{K}[x_{1},\dots,x_{c+(r-1)d}]$.
This kind of matrix is also called \emph{$d$-leap catalecticant} or \emph{$d$-catalecticant} in \cite{MR3275568}. And when $d=1$, it is precisely the ordinary \emph{Hankel matrix}, which is also known as a \emph{Toelitz matrix} in other fields like functional analysis, orthogonal polynomial theory, moment problem, and probability. Most importantly, Nam showed in \cite{Nam} that $I_{r}(\boldsymbol{H}_{r,c,d})$, the ideal of the maximal minors of $\boldsymbol{H}_{r,c,d}$, 
defines the $(r-1)$-th secant variety of the balanced rational normal scroll defined by $I_{2}(\bdH_{2,c+(r-2)d,d})$. 

Recall that  when $I$ is an ideal of $R$ minimally generated by some forms of the same degree, the \emph{Rees algebra} is $\mathcal{R}(I):=\oplus_{i\geq0}I^{i}t^{t}\subseteq R[t]$, and the \emph{fiber cone} is $\mathcal{F}(I)=\mathcal{R}(I)\otimes_R \mathbb{K}\cong\mathbb{K}[I]\subseteq R$, where $t$ is a new variable. As mentioned above, Conca, Herzog, and Valla \cite{Sagbi} gave the defining equations of  $\mathcal{R}(I)$ and
$\mathcal{F}(I)$ when $I=I_2(H_{2,c,d})$ is the defining ideal of a balanced rational normal scroll, writing explicitly these blow-up algebras as the quotient rings of some polynomial rings. Furthermore, they showed that $\mathcal{R}(I)$ and $\mathcal{F}(I)$ are Cohen--Macaulay normal domains in this case. In addition, if $\Char(\KK)=0$, then $\mathcal{R}(I)$ and $\mathcal{F}(I)$ have rational singularities. And if $\Char(\KK)>0$, then $\mathcal{R}(I)$ and $\mathcal{F}(I)$ are $F$-rational.

In this paper, we will study the blow-up algebras of the secant varieties of the balanced rational normal scrolls. In section 2, we are able to apply the Sagbi basis theory to obtain the Gr\"obner bases of the defining ideals of those blow-up algebras (\Cref{cor:lifting-the-toric-part} and \Cref{thm:defining-eqn-Hankel}). Consequently, the fiber cone $\mathcal{F}(I_r(\bdH_{r,c,d}))$ is a normal Cohen--Macaulay domain. And the singularity results also follow. By these, we extend the corresponding results of Conca, Herzog, and Valla in \cite{Sagbi} as well. Meanwhile, we are able to characterize when $\calF(I_r(\bdH_{r,c,d}))$ is Gorenstein in \Cref{Deformation}.

In section 3, we use the Sagbi deformation and the explicit description of the initial ideal to compute the Castelnuovo--Mumford regularity of the fiber cone (\Cref{thm:main-reg-balanced}). 
The formula presented there is rather dazzling, let alone its more involved proof. To be brief, we first associate the generators of the Alexander dual of the initial ideal of the defining ideal, with the maximal cliques of some graph $\mathcal{G}$ in subsection \ref{ss-Generators}. 
The main combinatorial tool is the moving sequence introduced in \Cref{lambdaIncrease}. This notion is essential for investigating the generators of the Alexander dual ideal comprehensively.
Then, we consider simultaneously the lexicographic and reverse lexicographic types of ordering with respect to
the ground ring $R=\mathbb{K}[x_{1},\dots,x_{c+(r-1)d}]$. We will show that both orderings give complete-intersection-quotients structures of the Alexander dual ideal in subsection \ref{ss:linear-quotients}. More precisely, each of the successive colon ideals is either a collection of single variables (corner generators) or a collection of
single variables with one extra higher degree monomial (tail generator). Even with complete intersection quotients, we only obtain a good upper bound on the regularity so far. 
In order to write down the precise value, we use the mapping cones to extract the projective dimension of the Alexander dual ideal, which in turn gives the regularity of the fiber cone. 
However, due to the existence of the tail generator, we don't have linear quotients in general, i.e., the mapping cones won't give a minimal free resolution of the Alexander dual ideal. Thus, we need to apply lots of careful observations, intricate techniques, and relatively prudent strategies during the process. The biggest obstacle lies in the relatively degenerated case when the number of columns $c$ of the matrix is smaller relative to the number of rows $r$ and the leaping distance $d$. Whence, the maximal length of the complete intersection quotients is much harder to justify; see subsection \ref{sec:sharp-bound-inter}. To overcome it, we  will create a balance by employing both lexicographic and reverse lexicographic types of ordering simultaneously. 
As an application,  we present $\bda$-invariant of the fiber cone and the reduction number of the ideal defined by the secant variety in \Cref{reduction}.

Our main results are summarized in the following.
\begin{Theorem}
    \begin{enumerate}[1] 
        \item The defining ideals of $\mathcal{R}(I_{r}(\boldsymbol{H}_{r,c,d}))$ and $\mathcal{F}(I_{r}(\boldsymbol{H}_{r,c,d}))$ are generated by the lifting of Sagbi bases, and those equations form Gr\"obner bases of these defining ideals.
        \item The Rees algebra $\mathcal{R}(I_{r}(\boldsymbol{H}_{r,c,d}))$ is of fiber type. In other words, its defining ideal is generated by the relations of the symmetric algebra of $I_r(\boldsymbol{H}_{r,c,d})$ together with the defining equations of the fiber cone $\mathcal{F}(I_{r}(\boldsymbol{H}_{r,c,d}))$.
        \item The fiber cone $\mathcal{F}(I_{r}(\boldsymbol{H}_{r,c,d}))$ has rational singularities in characteristic zero, and it is $F$-rational in positive characteristic. In particular, it is a Cohen--Macaulay normal domain.
        \item When $r\ge 2$ and $d\ge 1$, the fiber cone $\mathcal{F}(I_{r}(\boldsymbol{H}_{r,c,d}))$ is Gorenstein if and only if     
            \[
                c\in \Set{r,r+1,r+d,r+d+1,2r+d}.
            \]
        \item Closed formulas of the Castelnuovo--Mumford regularity and the $\bda$-invariant of the fiber cone $\calF(I_r(\bdH_{r,c,d}))$, as well as the reduction number of $I_r(\bdH_{r,c,d})$, are given.
    \end{enumerate}
\end{Theorem}

\section{Defining equations}
\label{sec:balanced-cases}


In this section, we focus on presenting the implicit defining equations of blow-up algebras. Throughout this section and the next one, $d$ will be a fixed positive integer. And the main object will be the \emph{extended Hankel matrix} $\bdH_{r,c,d}$, presented in equation \eqref{HMatrix}. The ground ring will usually be $R=\KK[\bdx]=\KK[x_1,\dots,x_{c+(r-1)d}]$ with $2\le r\le c$.

As mentioned in the introduction, the special ideal $I_2(\bdH_{2,c,d})$ gives the defining ideal of a balanced national normal scroll. And more generally, when $r>2$, $I_{r}(\boldsymbol{H}_{r,c,d})$ defines the $(r-1)$-th secant variety of the rational normal scroll defined by $I_{2}(\bdH_{2,c+(r-2)d,d})$ by \cite[Corollary 3.9]{Nam}.
Since the paper \cite{Nam} provides many interesting tools and results that we shall apply here, we will provide a succinct review in the following.

Denote by $>_{\lex}$ the lexicographic monomial order on $R$ induced by the order of the variables $x_1 > x_2 > \cdots > x_N$, where
\[
    N=N(r,c,d)\coloneqq \dim(R)=c+(r-1)d.
\]
We will only use this term order on $R$.
It is clear that every maximal minor of $\bdH_{r,c,d}$ can be uniquely determined by the indices of the elements on its main diagonal:  $\alpha_1<\alpha_2<\cdots<\alpha_r$. Thus, we will denote this minor by $M(\bdalpha)=M(\alpha_1,\alpha_2,\dots,\alpha_r)$. 
For instance, the minor using the first $r$ columns will be $M(1,2+d,3+2d,\dots,r+(r-1)d)$.
Notice that in this increasing sequence, one has $\alpha_i+d<\alpha_j$ for $i=1,2,\dots,r-1$. This leads to the following partial order $<_d$ on the set of positive integers:
\[
    i<_d j\quad \text{ if and only if }\quad i+d< j.
\]
We say that a sequence of positive integers $\alpha_1, \alpha_2 , \dots,\alpha_s$ is a \emph{$<_d$-chain} if $\alpha_1<_d \alpha_2<_d\cdots<_d \alpha_s$. Similarly, we say that a monomial $x_{\alpha_1}\cdots x_{\alpha_s}$ is a $<_d$-chain if its indices, when organized increasingly, form a $<_d$-chain. Note that a monomial of degree $r$ in $R$ is a $<_d$-chain if and only if it is the initial monomial of a maximal minor of $\bdH_{r,c,d}$.

In the following, we will also introduce 
\[
    \Lambda_{r,d}(N)\coloneqq
    \Set{\bdalpha=(\alpha_1,\dots,\alpha_r)\mid 1\le \alpha_1 <_d \alpha_{2} <_d \cdots <_d \alpha_r \le N}.
\] 
For each $\bdalpha=(\alpha_1,\dots,\alpha_r)\in \Lambda_{r,d}(N)$, we also write $\bdx_{\bdalpha}\coloneqq x_{\alpha_1}\cdots x_{\alpha_r}$. It is clear that $I=\Braket{M(\bdalpha)\mid \bdalpha\in \Lambda_{r,d}(N)}$.

\begin{Proposition}
    [{\cite[Corollary 3.9]{Nam}}]
    \label{claim-1}
    With respect to the lexicographic order on $R$, the maximal minors of $\bdH_{r,c,d}$ provide a Gr\"obner basis of the ideal $I$. In particular, the initial ideal $\ini_{>_{\lex}}(I)$ is generated by
    \begin{equation}
        G_{r,c,d}\coloneqq \Set{\bdx_{\bdalpha}\mid \bdalpha\in \Lambda_{r,d}(N)}.
        \label{eqn:Grcd}
    \end{equation}
\end{Proposition}  

In addition to \Cref{claim-1}, one actually has the following much stronger result.

\begin{Proposition}
    [{\cite[Theorem 3.28(b)]{Nam}}]
    \label{claim-2}
   
    One has $(\ini_{>_{\lex}}(I))^k=\ini_{>_{\lex}}(I^k)$ for all positive integer $k$.
\end{Proposition}

Temporarily, we fix a term order $\tau$ on the monomials in $R$. Recall that if $A$ is a finitely generated $\KK$-subalgebra of $R$, the \emph{initial algebra} of $A$, denoted by $\ini_{\tau}(A)$, is the $\KK$-subalgebra of $R$ generated by the initial monomials $\ini_{\tau}(a)$ for all $a\in A$. And a set of elements $a_i\in A$, $i\in \calI$, is called a \emph{Sagbi basis} if $\ini_{\tau}(A)=\KK[\ini_{\tau}(a_i)\mid i\in \calI]$.
The terminology ``Sagbi'' is the acronym for ``\underline{S}ubalgebra \underline{a}nalog to \underline{G}r\"obner \underline{b}ases for \underline{i}deal''. In this work, the term order $\tau$ in mind is the lexicographic order $>_{\lex}$.

We need to extend the term order $>_{\lex}$ on $R$ to the order $\succ$ on $R[t]=\KK[\bdx,t]$ as follows: for two monomials $\bdx^{\bda}t^i$ and $\bdx^{\bdb}t^j$ of $R[t]$, set $\bdx^{\bda}t^i\succ\bdx^{\bdb}t^j$ if $i>j$ or $i=j$ and $\bdx^{\bda}>_{\lex}\bdx^{\bdb}$ with respect to the lexicographic order. Here, $\bdx^\bda\coloneqq x_1^{a_1}\cdots x_N^{a_N}$ for $\bda=(a_1,\dots,a_N)\in \NN^N$, and $\bdx^\bdb$ is similarly defined. Using \Cref{claim-2} and \cite[Theorem 2.7]{Sagbi}, one obtains the equality $\calR(\ini_{>_{\lex}}(I))=\ini_{\succ}(\calR(I))$.
One can then verify that $\{x_1,\dots,x_N\}\cup\Set{M(\bdalpha)t\mid \bdalpha\in \Lambda_{r,d}(N)}$ forms a Sagbi basis of the Rees algebra $\calR(I)=R[It]=\KK[x_1,\dots,x_N,It]\subset R[t]$. 

The most pleasant thing here is that we can use $\calR(\ini_{>_{\lex}}(I))$ to study the Rees algebra $\calR(I)$ via the techniques in \cite[Section 2]{Sagbi}. For $\calR(\ini_{>_{\lex}}(I))$, we consider the following canonical epimorphism 
\[
    R[\bdY]\coloneqq R[Y_\bdalpha\mid \bdalpha\in \Lambda_{r,d}(N)] \to \calR(\ini_{>_{\lex}}(I)), \quad Y_\bdalpha\mapsto \ini_{>_{\lex}}(M(\bdalpha))t=\bdx_{\bdalpha}t.
\]
The kernel of this homomorphism will be called the \emph{defining ideal} of the Rees algebra $\calR(\ini_{>{\lex}}(I))$.

\begin{Proposition}
    [{\cite[Proposition 5.12 and Theorem 5.13]{Nam}}] 
    \label{ReesNoraml}
    The Rees algebra $\calR(I)= \bigoplus_{i\ge 0} I^i t^i \subset R[t]$ is a normal Cohen--Macaulay Koszul domain, defined by a Gr\"obner basis of quadratics.   Furthermore, the initial algebra $\ini_{\succ}(\calR(I))=\calR(\ini_{>_{\lex}}(I))$ is defined by a Gr\"obner basis of quadratics such that the underlined part will give the leading monomial:
    \begin{enumerate}[i]
        \item \label{item-toric-part} $\underline{Y_\bdalpha Y_\bdbeta} - Y_{\bdalpha'}Y_{\bdbeta'}$: $(\bdalpha',\bdbeta')$ is the quasi-sorted pair reduction of $(\bdalpha,\bdbeta)$;
        \item \label{item-syzygy-part} $\underline{x_t Y_\bdalpha} - x_{a_k} Y_\bdbeta$: with $\alpha_{k-1} <_d t < \alpha_k$ for some $k$ and $\bdbeta\coloneqq (\alpha_1, \dots , \alpha_{k-1}, t, \alpha_{k+1}, \dots , \alpha_r)$ by assuming that $a_0 = -\infty$.
    \end{enumerate}
\end{Proposition}
In particular, the quadratic equations in item \ref{item-toric-part} form a Gr\"obner basis of the defining ideal of the fiber cone of $\ini_{>_{\lex}}(I)$ by \cite[Theorem 5.14]{Nam}. Actually, the reduction process that Nam applied in \cite[Section 5]{Nam} aims at the more complicated multi-Rees algebra. However, when focusing on the Rees algebra here, it can be greatly simplified as follows to give a reduced Gr\"obner basis of quadratics.

First of all, from the algorithm stated in \cite[Algorithm 5.9]{Nam}, it is clear that the quasi-sorted pair $(\bdalpha',\bdbeta')$ in the item \ref{item-toric-part} is actually sorted in the sense of \cite[Definition 5.1]{Nam}. The notation of \emph{sortedness} has actually been studied earlier. Let $\Mon_r$ be the set of monomials of degree $r$ in $R=\KK[x_1,\dots,x_N]$.
We will apply the \emph{sorting operator}
\[
    \sort: \Mon_r\times \Mon_r\to \Mon_r\times \Mon_r,\quad (u,v)\mapsto (u',v')
\]
considered in \cite[Chapter 14]{Sturmfels}. Recall that if $u$ and $v$ are two monomials in $\Mon_r$ such that $uv=x_{i_1}x_{i_2}\cdots x_{i_{2r}}$ with $i_1\le i_2\le \cdots\le i_{2r}$, then $u'=x_{i_1}x_{i_3}\cdots x_{i_{2r-1}}$ while $v'=x_{i_2}x_{i_4}\cdots x_{2r}$. And the pair $(u,v)$ is called \emph{sorted} if $\sort(u,v)=(u,v)$.

The existence of the generators in item \ref{item-toric-part} actually implies the following fact.

\begin{Lemma}
    \label{Lem:is-sortable}
    The set $G_{r,c,d}$ defined in \eqref{eqn:Grcd} is {sortable}, i.e., 
    \[
        \sort(G_{r,c,d}\times G_{r,c,d})\subseteq G_{r,c,d}\times G_{r,c,d}.
    \]
\end{Lemma}

Conversely, by {\cite[Theorem 6.16]{EH}}, the set $G_{r,c,d}$ being sortable also implies that the generators in item \ref{item-toric-part} is the reduced Gr\"obner basis of the defining ideal of $\KK[\ini_{>_{\lex}}(I)]\cong \calF(\ini_{>_{\lex}}(I))$, and the underlined part is the corresponding leading monomial. \Cref{Lem:is-sortable} has been verified directly in \cite[Proposition 3.1]{MR4019342}.

Now, we are ready to state the following straightening-law type of result.

\begin{Proposition}
    \label{thm:Finer-straightening-law}
    Let $M(\bdalpha)$ and $M(\bdbeta)$ be two arbitrary maximal minors in $I=I_r(\bdH_{r,c,d})$. Then there exist $\bdalpha_1,\dots,\bdalpha_n,\bdbeta_1,\dots,\bdbeta_n\in \Lambda_{r,d}(N)$ and $\mu_1,\dots,\mu_n\in \KK$, such that
    \begin{equation}
        M(\bdalpha)M(\bdbeta)=\sum_{i=1}^{n} \mu_i M(\bdalpha_i)M(\bdbeta_i),
        \label{eqn:SL}
    \end{equation}
    with
    \begin{align}
        \ini_{>_{\lex}}(M(\bdalpha)M(\bdbeta))&=
        \ini_{>_{\lex}}(M(\bdalpha_1)M(\bdbeta_1)) \notag \\
        &>\ini_{>_{\lex}}(M(\bdalpha_2)M(\bdbeta_2))>\cdots>
        \ini_{>_{\lex}}(M(\bdalpha_n)M(\bdbeta_n)),
        \label{eqn:ini-increasing}
    \end{align}
    and
    \begin{equation}
        \text{the pair $(\bdx_{\bdalpha_i},\bdx_{\bdbeta_i})$ is sorted for each $i$.}
        \label{eqn:all-sorted}
    \end{equation}
\end{Proposition}

\begin{proof}
    Let $f_0=M(\bdalpha)M(\bdbeta)\in I^2$.  Suppose that $\sort(\bdalpha,\bdbeta)=(\bdalpha_1,\bdbeta_1)$. By \Cref{Lem:is-sortable}, $\bdalpha_1,\bdbeta_1\in \Lambda_{r,d}(N)$.  Now, we write $f_1\coloneqq f_0-M(\bdalpha_1)M(\bdbeta_1)\in I^2$ and choose $\mu_1=1$. If $f_1\ne 0$, by \Cref{claim-2}, we can find $\bdalpha_1',\bdbeta_1'\in \Lambda_{r,d}(N)$ and $\mu_2\in \KK$ such that $\ini(f_1)=\mu_2 \bdx_{\bdalpha_1'}\bdx_{\bdbeta_1'}$. Now, as above, we can find $\bdalpha_2,\bdbeta_2\in \Lambda_{r,d}(N)$ with $\sort(\bdalpha_1',\bdbeta_1')=(\bdalpha_2,\bdbeta_2)$ and make $f_2\coloneqq f_1-\mu_2 M(\bdalpha_2)M(\bdbeta_2)\in I^2$. If $f_2\ne 0$, we will keep the construction process. This process, however, will terminate due to the monomial ordering. And the properties stated in the proposition are clear.
\end{proof}

\begin{Remark}
    The straightening law with only \eqref{eqn:SL} and \eqref{eqn:ini-increasing} above is hardly surprising. The real contribution here is the tail part in \eqref{eqn:ini-increasing} as well as \eqref{eqn:all-sorted}. 

    Notice that with the lexicographic order on $R$, the initial monomial of each $f_i$ is easy to obtain. Since the sorting operator only cares about the product of the inputting pair, we can apply it to $\ini_{>_{\lex}}(f_i)$ and find the next product of maximal minors with ease. And \Cref{Lem:is-sortable} guarantees that this operation is well-defined.
\end{Remark}

Nam showed already in \cite{Nam} that the initial algebra of the Rees algebra is actually the Rees algebra of the initial ideal, namely $\ini \calR(I)=\calR(\ini I)$. But he didn't consider explicitly $\ini \calF(I)$, the initial algebra of the fiber cone. And the defining equations of the blow-up algebras $\calR(I)$ and $\calF(I)$ are not described explicitly in \cite{Nam}. 
Not only that, he did not study explicitly the Sagbi basis of the fiber cone $\calF(I)$. 
We will handle these vacancies in the following.

\begin{Theorem}
    \label{Fiber-Sagbi}
    The set $\Set{M(\bdalpha)\mid \bdalpha\in \Lambda_{r,d}(N)}$ forms a Sagbi basis of $\mathcal{F}(I)$, i.e., one has $\ini_{>_{\lex}} (\mathcal{F}(I))=\mathcal{F}(\ini_{>_{\lex}}(I))$.
\end{Theorem}

\begin{proof}
    We know that $\Set{\underline{Y_\bdalpha Y_\bdbeta} - Y_{\bdalpha_1}Y_{\bdbeta_1}\mid (\bdx_{\bdalpha_1},\bdx_{\bdbeta_1})=\sort(\bdx_{\bdalpha},\bdx_{\bdbeta})}$ gives a generating set of the defining ideal of the fiber cone $\mathcal{F}(\ini_{>_{\lex}} (I))\cong \KK[\ini_{>_{\lex}}(M(\bdalpha))\mid \bdalpha\in \Lambda_{r,d}(N)]$ by \cite[Theorem 5.14]{Nam}. Equation \eqref{eqn:SL} in \Cref{thm:Finer-straightening-law} gives  $M(\bdalpha)M(\bdbeta)-M(\bdalpha_1)M(\bdbeta_1)=\sum_{i=2}^{n} \mu_i M(\bdalpha_i)M(\bdbeta_i)$. Now with \eqref{eqn:ini-increasing}, we can apply \cite[Proposition 1.1]{Sagbi}. 
\end{proof}

\begin{Corollary}
     \label{cor:lifting-the-toric-part}
    The defining ideal $P\subset \KK[Y_\bdgamma \mid \bdgamma \in G_{r,c,d}]$ of the fiber cone $\KK[I]$ has a reduced Gr\"obner basis given by
    \begin{equation}
        Y_\bdalpha Y_\bdbeta-\sum_{i=1}^{n} \mu_i Y_{\bdalpha_i}Y_{\bdbeta_i}
        \label{eqn:quadratic-relations}
    \end{equation}
    for all unsorted pair $(\bdalpha,\bdbeta)$ in $G_{r,c,d}\times G_{r,c,d}$ with $\bdx_{\bdalpha}>_{\lex} \bdx_{\bdbeta}$, such that
    \[
        M(\bdalpha)M(\bdbeta)=\sum_{i=1}^{n} \mu_iM(\bdalpha_i)M(\bdbeta_i)
    \]
    as in \eqref{eqn:SL}. Moreover, $\Set{Y_\bdalpha Y_\bdbeta| \text{$(\bdx_{\bdalpha}>_{\lex}\bdx_{\bdbeta})$ is unsorted}}$ is the generating set of the initial ideal of $P$ with respect to some term order $\tau$.   
\end{Corollary}  

\begin{proof}
  
    It follows from \Cref{Fiber-Sagbi} and \cite[Corollary 2.1]{Sagbi} that the polynomials in \eqref{eqn:quadratic-relations} will generate the defining ideal $P$. As for the remaining part, it suffices to apply \Cref{Lem:is-sortable}, {\cite[Theorem 6.16]{EH}} and \cite[Corollary 2.2]{Sagbi}.
\end{proof}

Similar to \Cref{thm:Finer-straightening-law} and \Cref{cor:lifting-the-toric-part}, we can apply \cite[Corollary 2.2]{Sagbi} to give a primitive lifting of the binomial for the syzygy part in item \ref{item-syzygy-part} of \Cref{ReesNoraml}. However, there is a more elegant way to do it, i.e., we can borrow the strength from the Eagon--Northcott resolution in \cite{Eagon-Northcott}.

\begin{Theorem}
    \label{thm:defining-eqn-Hankel}
    The relations in \eqref{eqn:quadratic-relations} together with the syzygy-type relations
    \[
        \sum_{j=1}^{r+1} (-1)^{j+1} x_{c_{j+(k-1)d}}
        Y_{{c_1},{c_2+d},\dots,{c_{j-1}+(j-2)d},{c_{j+1}+(j-1)d},\dots,{c_{r+1}+(r-1)d}}
    \]
    for arbitrary $c_1<c_2<\cdots<c_{r}<c_{r+1}$  and arbitrary $k$ with $1\le k \le r$, form a Gr\"obner basis of the defining ideal of $\calR(I)$ with respect to some term order on $R[\bdY]$.
\end{Theorem}

\begin{proof}
    The presentation matrix of $I$ can be obtained from the Eagon--Northcott resolution. We describe it in detail here. Take arbitrary increasing column indices $c_1<c_2<\cdots<c_{r}<c_{r+1}$ and arbitrary $k$ with $1\le k\le r$. We will consider the determinant of the $(r+1)\times(r+1)$ matrix $\bdA$ pictured in \Cref{fig:matrix-A}, where the $k$-th row and the $(k+1)$-th row are identical.
   
    \begin{figure}[htb]
        $\bdA=
        \begin{pNiceMatrix}
            x_{c_1}        & \Cdots & x_{c_k}        & x_{c_{k+1}}        & \Cdots & x_{c_{r+1}}        \\
            x_{c_1+d}      & \Cdots & x_{c_k+d}      & x_{c_{k+1}+d}      & \Cdots & x_{c_{r+1}+d}      \\
            \Vdots         &        & \Vdots         & \Vdots             &        & \Vdots             \\
            x_{c_1+(k-1)d} & \Cdots & x_{c_k+(k-1)d} & x_{c_{k+1}+(k-1)d} & \Cdots & x_{c_{r+1}+(k-1)d} \\
            x_{c_1+(k-1)d} & \Cdots & x_{c_k+(k-1)d} & x_{c_{k+1}+(k-1)d} & \Cdots & x_{c_{r+1}+(k-1)d} \\
            x_{c_1+kd}     & \Cdots & x_{c_k+kd}     & x_{c_{k+1}+kd}     & \Cdots & x_{c_{r+1}+kd}     \\
            \Vdots         &        & \Vdots         & \Vdots             &        & \Vdots             \\
            x_{c_1+(r-1)d} & \Cdots & x_{c_k+(r-1)d} & x_{c_{k+1}+(r-1)d} & \Cdots & x_{c_{r+1}+(r-1)d}
        \end{pNiceMatrix}$
        \caption{The matrix $\bdA$} 
        \label{fig:matrix-A}
    \end{figure}
    
    By expanding along the $k$-th row, we arrive at the relation
    \[
        \sum_{j=1}^{r+1} (-1)^{j+1} x_{c_{j+(k-1)d}}\det(\bdA_{c_1,\dots,\hat{c_j},\dots,c_{r+1}})=0,
    \]
    where $\bdA_{c_1,\dots,\hat{c_j},\dots,c_{r+1}}$ is obtained from the matrix $\bdA$ by removing the $k$-th row and the $j$-th column.

    It is clear that the leading monomial of $x_{c_{j+(k-1)d}}\det(\bdA_{c_1,\dots,\hat{c_j},\dots,c_{r+1}})$ is
    \begin{equation}
        x_{c_{j+(k-1)d}}x_{c_1}x_{c_2+d}\cdots x_{c_{j-1}+(j-2)d}x_{c_{j+1}+(j-1)d}\cdots x_{c_{r+1}+(r-1)d},
        \label{eqn:leading-monomial}
    \end{equation}
    which is a term, up to sign, of the full expansion of $\det(\bdA)$, before cancellation.

    On the other hand, as the indices of the entries are non-decreasing along the rows and columns of $\bdA$, the lexicographic order guarantees that the leading monomial in the full expansion is the diagonal one:
    \[
        x_{c_1}x_{c_2+d}\cdots x_{c_{k}+(k-1)d}x_{c_{k+1}+(k-1)d}\cdots x_{c_{r+1}+(r-1)d}.
    \]
    It is precisely when we choose $j=k$ or $j=k+1$ in \eqref{eqn:leading-monomial}. This corresponds to the binomial in item \ref{item-syzygy-part} given by Nam, after making  
    \[
        (c_1,c_2+d,\dots,c_k+(k-1)d,c_{k+1}+(k-1)d,\dots,c_{r+1}+(r-1)d)=
        (\alpha_1,\dots,\alpha_{k-1},t,\alpha_k,\dots,\alpha_r).
    \]

    With \Cref{cor:lifting-the-toric-part} in mind, the last step is to apply the proof of \cite[Theorem 5.13]{Nam} and \cite[Proposition 1.1 and Corollares 2.1, 2.2]{Sagbi}.
\end{proof}

Since we know the Sagbi bases of the fiber cone $\calF(I)$ and the Rees algebra $\calR(I)$, 
we can end this section with some quick applications.

\begin{Proposition}
    \label{Deformation}
    \begin{enumerate}[a]
        \item \label{Deformation-a}
        The fiber cone $\calF(I)$ is a normal Cohen--Macaulay domain. In particular, $\calF(I)$ has rational singularities if $\Char(\mathbb{K})= 0$, and it is F-rational if $\Char(\mathbb{K})>0$.
        \item When $r\ge 2$ and $d\ge 1$, the fiber cone $\calF(I)$ is Gorenstein if and only if     
            \[
                c\in \Set{r,r+1,r+d,r+d+1,2r+d}.
            \]
        \item \label{Deformation-c}
        The analytic spread of the ideal $I$, namely, the dimension of the fiber cone is
            \begin{equation}
                \dim(\calF(I))= 
                \begin{cases}
                    c+(r-1)d, & \text{if $ r+d<c$}, \\
                    rc-r^2+1, & \text{if $r<c\le r+d$},\\
                    1, & \text{if $r=c$}.
                \end{cases}
                \label{eqn:dim-toric-ring}
            \end{equation}
    \end{enumerate}  
\end{Proposition}

\begin{proof}
    Nam in \cite[Theorem 3.28 and Proposition 5.12]{Nam} showed that  $\calR(I)$ is normal, and $\ini(\calR(I))=\calR(\ini(I))$. Now, $\calF(\ini(I))$ inherits the normality from $\calR(\ini(I))$ by \cite[Theorem 7.1]{MR1283294}. 
    The deformation theory of the Sagbi basis then gives the normality, Cohen--Macaulayness, and singularity of $\calF(I)$ by \Cref{Fiber-Sagbi} and \cite[Corollary 2.3]{Sagbi}. And the Gorensteinness of the fiber cone is the by-products of Sagbi basis theory with \cite[Corollary 2.7]{arXiv1805.11923} and \cite[Theorem 3.7]{arXiv:1901.01561}.
    
    It remains to prove the item \ref{Deformation-c}.
    The case when $c\ge r+d$ follows directly from \cite[Theorems 2.3]{arXiv:1901.01561}. 
    And the case when $c=r$ is due to the fact that $I$ is principal.
    As for the remaining case when $1<r<c<r+d$, by the reduction before \cite[Theorem 2.3]{arXiv:1901.01561}, we can reduce the ideal $I_r(\bdH_{r,c,d})\subset R=\KK[\bdx]$ to some $I_r(\bdH_{r,c,d'})\subset \KK[{\bdx'}]$ with $d'=c-r$. Here, the collection of variables ${\bdx'}$ is a subset of the original collection of variables $\bdx$, and $I_r(\bdH_{r,c,d'})\KK[\bdx]=I_r(\bdH_{r,c,d})$.
    Whence, $c=r+d'$ and by the result in the $c=r+d$ case, the current dimension is $rd'+1=rc-r^2+1$. 
\end{proof}

\section{Regularity of the fiber cone}

In this section, we focus on the Castelnuovo--Mumford regularity of the fiber cone $\calF(I)\cong \KK[I]$ for the ideal $I=I_r(\bdH_{r,c,d})\subset R=\KK[x_1,\dots,x_{c+(r-1)d}]$ where $2\le r\le c$. The computations and proofs are quite involved. We encourage the readers to refer back to the introduction for the road map of this section. 

Recall that the dimension of the ground ring $R$ is denoted by
\[
    N=N(r,c,d)\coloneqq \dim(R)=c+(r-1)d
\]
in this paper. And we have the natural epimorphisms
\begin{align}
    \KK[\bdY]=\KK[Y_{\bdalpha}\mid \bdalpha\in \Lambda_{r,d}(N)] &\onto \KK[I]=\KK[M(\bdalpha)\mid \bdalpha\in \Lambda_{r,d}(N)],
    \label{eqn:P_0} \\
    \intertext{and}
    \KK[Y_{\bdalpha}\mid \bdalpha\in \Lambda_{r,d}(N)] &\onto \KK[\ini(I)]=\KK[\ini(M(\bdalpha))\mid \bdalpha\in \Lambda_{r,d}(N)].
    \label{eqn:P_1}
\end{align}
Furthermore, if $\beta_{i,j}$ is the graded Betti number of $\KK[I]$ considered as a $\KK[\bdY]$-module, then the Castelnuovo--Mumford regularity of $\KK[I]$ is defined to be
\[
    \reg(\KK[I])\coloneqq \max_{i,j} \Set{j-i\mid \beta_{i,j}\ne 0}.
\]
The following main result of this section computes this regularity explicitly.

\begin{Theorem}
    \label{thm:main-reg-balanced}
    The Castelnuovo--Mumford regularities of the fiber cones $\calF(I_r(\bdH_{r,c,d}))$ and $\calF(\ini (I_r(\bdH_{r,c,d})))$ are given by
    \[
        \reg(\calF(I_r(\bdH_{r,c,d})))=\reg(\calF(\ini (I_r(\bdH_{r,c,d}))))=
        \begin{cases}
            N-1-\floor{(N-1)/r}, & \text{if $2r+d\le c$},\\
            dr-2r-3d+2c-2, & \text{if $r+d< c <2r+d$},\\
            (r-1)(c-r-1), & \text{if $r< c\le r+d$},\\
            0, & \text{if $r=c$}.
        \end{cases}
    \]
\end{Theorem}

\begin{proof}
    Let $P$ and $P'$ be the kernels of the epimorphisms in \eqref{eqn:P_0} and \eqref{eqn:P_1} respectively. By \Cref{cor:lifting-the-toric-part}, 
    \[
        \ini(P):=\ini_{\tau}(P)=\Braket{Y_\bdalpha Y_\bdbeta\mid \text{$(\bdx_{\bdalpha}>_{\lex}\bdx_{\bdbeta})$ is unsorted}}
    \]
    is squarefree. 
    Meanwhile, by \Cref{Fiber-Sagbi} and \cite[Corollary 2.2]{Sagbi}, there exists some term order $\tau'$ such that $\ini_{\tau}(P)=\ini_{\tau'}(P')$.
    Then, by our favorite \cite[Corollary 2.7]{arXiv1805.11923}, we will have 
    \[
        \reg(\KK[\bdY]/P)=\reg(\KK[\bdY]/\ini(P))=\reg(\KK[\bdY]/\ini_{\tau'}(P'))=\reg(\KK[\bdY]/P').
    \]
    Therefore, it amounts to explore the regularity of $\KK[\bdY]/\ini(P)$ considered as a $\KK[\bdY]$-module.

   Note that $\mathcal{F}(I)=\KK[I]\cong \KK[\bdY]/P$ is Cohen--Macaulay by \Cref{Deformation} \ref{Deformation-a}. Consequently, by \cite[Corollary 2.7]{arXiv1805.11923}, $\KK[\bdY]/\ini(P)$ is Cohen--Macaulay as well.

    Now, back to the probe of the regularity. The extremal situation when $r=c$ is easy since the ideal is principal. Meanwhile, the case when $r<c<r+d$ can be reduced to the case $c=r+d$ by the reduction stated before \cite[Theorem 2.3]{arXiv:1901.01561}; see also the explanation in the proof of \Cref{Deformation}.
  
    Now, we can focus on the cases when $c\ge r+d$. Whence,
    by applying \Cref{Deformation} (c), the Auslander--Buchsbaum theorem \cite[Theorem 1.3.3]{MR1251956} and the Eagon--Reiner theorem \cite[Theorem 8.1.9]{MR2724673}, we know that the Alexander dual of the squarefree ideal $\ini(P)$ has a linear resolution with regularity $\widetilde{N}-N$ or $\widetilde{N}-(rd+1)$ when $c>r+d$ or $c= r+d$ respectively.  Here, $\widetilde{N}=\dim(\KK[\bdY])$, namely the cardinality of $\Lambda_{r,d}(N)$.  Our task is then to find the projective dimension of $(\ini(P))^\vee$, by \cite[Proposition 8.1.10]{MR2724673}.
    This task can be completed by combining the coming Propositions \ref{thm-main-result}, \ref{prop:sharp-bound-inter} and \ref{prop:less-than-r+d}.  
\end{proof}

Here is the plan on how to complete the final task stated just now.
The concrete computation of $\projdim((\ini(P))^\vee)$ will be carried out for the case $2r+d\le c$ in subsection \ref{subsection:c>=2r+d}, for the case $r+d < c < 2r+d$ in subsection \ref{sec:sharp-bound-inter} and for the case $r\le c\le r+d$ in subsection \ref{sec:less-than-r+d}.
However, before that, we need to describe explicitly the minimal monomial generators of $(\ini(P))^\vee$ in subsection \ref{ss-Generators}. We will give a linear order on the minimal monomial generating set $G((\ini(P))^\vee)$, and show that $(\ini(P))^\vee$ has complete intersection quotients in subsection \ref{ss:linear-quotients}. The maximal length of the complete intersections in general only gives an upper bound of the projective dimension. By studying the iterated mapping cones carefully, we will show that the projective dimension is actually achieved by the maximal length. The general strategy will be stated in subsection \ref{ss:max-length}. And it is one of the essential parts for the computations in subsections \ref{subsection:c>=2r+d}, \ref{sec:sharp-bound-inter}, and \ref{sec:less-than-r+d}.

\subsection{Minimal monomial generators of $(\ini(P))^{\vee}$}
\label{ss-Generators}
In the following, we shall deal with two different linear order $>$ on the variables of $\KK[\bdY]=\KK[Y_{\bdalpha}\mid \bdalpha\in \Lambda_{r,d}(N)]$:
\begin{equation}
    \begin{minipage}{0.9\linewidth}
        \begin{enumerate}[i]
            \item \textbf{$lex$ type}:
                $Y_{\bdalpha}>Y_{\bdbeta} \Leftrightarrow \bdx_{\bdalpha}>_{\lex}\bdx_{\bdbeta}$, or
            \item \textbf{$revlex$ type}:
                $Y_{\bdalpha}>Y_{\bdbeta} \Leftrightarrow \bdx_{\bdalpha}>_{\revlex}\bdx_{\bdbeta}$.
        \end{enumerate}
    \end{minipage}
    \label{2-orderings}
\end{equation}
We will then consider the lexicographic order $>_{\lex}$ on $\KK[\bdY]$ with respect to either given linear order $>$ on the variables.

It is already known that a minimal monomial generating set of $\ini(P)=\ini_{\tau}(P)$ is given by
\[
    \Set{Y_{\bdalpha}Y_{\bdbeta}\mid \bdx_\bdalpha>_{\lex}\bdx_\bdbeta \text{ but $(\bdx_{\bdalpha},\bdx_{\bdbeta})$ is not sorted}}.
\]
It is not difficult to see that this is also
\[
    \Set{Y_{\bdalpha}Y_{\bdbeta}\mid \bdx_\bdalpha>_{\revlex}\bdx_\bdbeta \text{ but $(\bdx_{\bdalpha},\bdx_{\bdbeta})$ is not sorted}}.
\]
Let $\calG$ be the simple graph on the vertex set $\Set{Y_{\bdalpha}\mid \bdalpha\in \Lambda_{r,d}(N)}$, such that $\Set{Y_{\bdalpha}>Y_{\bdbeta}}$ is an edge if and only if $(\bdx_{\bdalpha},\bdx_{\bdbeta})$ is \emph{sorted}. It is then clear that $\ini(P)$ is the edge ideal of the complement graph $\calG^{\complement}$.

In the remaining of this subsection, we will only consider the case when 
$c> r+d$; the case when $r\le c\le r+d$, which is only slightly different, will be left in the final subsection \ref{sec:less-than-r+d}. Whence, as the ideal $(\ini(P))^{\vee}$ 
has a linear resolution by the proof of \Cref{thm:main-reg-balanced}, it is minimally generated by some squarefree monomials of degree
\[
    \reg((\ini(P))^{\vee})=\widetilde{N}-N
\]
by \eqref{eqn:dim-toric-ring}.

Suppose that $Y_{\bdalpha_1}Y_{\bdalpha_2}\cdots Y_{\bdalpha_{\widetilde{N}-N}}\in G(\ini(P)^{\vee})$ is a minimal monomial generator. This is equivalent to saying that $\Set{Y_{\bdalpha_1},\dots,Y_{\bdalpha_{\widetilde{N}-N}}}$ is a minimal vertex cover of the complement graph $\calG^{\complement}$, by \cite[Corollary 9.1.5]{MR2724673}. In other words, this means that
\[
    \Set{Y_{\bdbeta_1},\dots,Y_{\bdbeta_N}}\coloneqq \Set{Y_{\bdalpha_1},\dots,Y_{\bdalpha_{\widetilde{N}-N}}}^{\complement} 
\]
is a maximal clique of $\calG$. Notice that the cardinality of this set is precisely $N$.

\begin{Notation}
    \begin{enumerate}[1]
        \item For any squarefree monomial $u\in \KK[\bdY]$, write 
            \[
                \widehat{u}\coloneqq \left.\left(\prod_{\bdalpha\in \Lambda_{r,d}(N)}Y_{\bdalpha}\right)\right/u.
            \]
            On the other hand, if $F$ is a subset of the vertex set of $\calG$, we will write 
            \[
                \bdY^F \coloneqq \prod_{f\in F}f.
            \]
        \item For any increasing sequence $\bdbeta=(b_1<b_2<\cdots<b_s)$, we will write $\bdbeta^k=b_k$ for $k=1,2,\dots,s$. We will also use $\bdbeta^{\le k}$ for the subsequence $(b_1<b_2<\cdots<b_k)$ and similarly define $\bdbeta^{\ge k}$. Meanwhile, $\bdbeta^{\{k_1,...,k_2\}}$ will be the subsequence $(b_{k_1}<b_{k_1+1}<\cdots<b_{k_2})$.
    \end{enumerate}
\end{Notation}

Therefore, the minimal monomial generating set of $\ini(P)^{\vee}$ is given by
\begin{equation}
    \Set{\widehat{\bdY^F}| \text{$F$ is a maximal clique of $\calG$}}.
    \label{eqn:max-clique-gen}
\end{equation} 
Next, we describe the maximal cliques of $\calG$. 

\begin{Lemma}
    \label{MaxCli}
    The set
    \begin{equation}
        F=\Set{Y_{\bdbeta_1}>\cdots>Y_{\bdbeta_N}}
        \label{eqn:F-beta}
    \end{equation}
    is a maximal clique of $\calG$ if and only if $F$ satisfies the following three groups of conditions:
    \begin{gather}
        \scalebox{0.96}{$
            1\le \bdbeta_{1}^{1}\le \bdbeta_{2}^{1}\le \cdots \le \bdbeta_{N}^{1}\le
            \bdbeta_{1}^{2}\le \bdbeta_{2}^{2}\le \cdots \le \bdbeta_{N}^{2}\le  \cdots \le
            \bdbeta_{1}^{r}\le \bdbeta_{2}^{r}\le \cdots \le \bdbeta_{N}^{r}\le N;
            $}
        \label{eqn:clique-condition} \\
        \bdbeta_{k}^{j_k}+1=\bdbeta_{k+1}^{j_k} \text{ for some $j_k$, while } \bdbeta_{k}^{j}=\bdbeta_{k+1}^{j}\text{ for $j\ne j_k$};
        \label{condition:move-1} \\
        \bdbeta_{1}^{1}=1, \quad
        \bdbeta_{N}^{\le r-1}=\bdbeta_{1}^{\ge 2} \quad\text{and}\quad
        \bdbeta_{N}^{r}=N.
        \label{condition:clique-end}
    \end{gather} 
\end{Lemma}

\begin{proof}
    If $F$ is a maximal clique of $\calG$, then the pair $(\bdx_{\bdbeta_i},\bdx_{\bdbeta_j})$ is sorted whenever $1\le i<j \le N$. In other words, we will obtain the condition (\ref{eqn:clique-condition}). Since the clique $F$ is maximal, for each $k=1,2,\dots,N-1$, there exists a unique $j_k$ such that the condition (\ref{condition:move-1}) is met. Consequently, we have
    \begin{align*}
        N-1&=\sum_{j=1}^r(\bdbeta_N^j-\bdbeta_1^j)=\sum_{j=1}^{r-1}(\bdbeta_N^j-\bdbeta_1^j)+(\bdbeta_N^r-\bdbeta_1^r)\\
        & \le \sum_{j=1}^{r-1}(\bdbeta_1^{j+1}-\bdbeta_1^j)+(N-\bdbeta_1^r)= N-\bdbeta_1^1 \le N-1.
    \end{align*}
    Therefore, we obtain the condition (\ref{condition:clique-end}).

    On the other hand, it is not difficult to see that the conditions \eqref{eqn:clique-condition},  \eqref{condition:move-1}, and  \eqref{condition:clique-end} together are sufficient for the set $F$ in \eqref{eqn:F-beta} to be a maximal clique of $\calG$. 
\end{proof}

The following remark gives us a better view of the maximal clique with respect to the running indices. 

\begin{Remark}
    \label{rmk:moving-seq}
    Notice that given the initial sequence $\bdbeta_1$, the maximal clique $F$ in \eqref{eqn:F-beta} is completely determined by the \emph{moving sequence}:
    \begin{equation}
        \bdlambda_F= (j_1,j_2,\dots,j_{N-1}),
        \label{eqn:moving-seq}
    \end{equation}
    where the $j_k$'s are given by the condition \eqref{condition:move-1}. We will call the integer $j_k$ as the \emph{movement} from $Y_{\bdbeta_k}$ to $Y_{\bdbeta_{k+1}}$ and call it a \emph{$j_k$-movement} in $\bdlambda_F$.

    On the other hand, notice that the cardinalities
    \begin{equation}
        \left|\Set{k\mid j_k=t}\right|=\bdbeta_1^{t+1}-\bdbeta_1^t \qquad \text{for $t=1,2,\dots,r-1$,}
        \label{eqn:mv-determine-beta1-1}
    \end{equation}
    and
    \begin{equation}
        \left|\Set{k\mid j_k=r}\right|=N-\bdbeta_1^r,
        \label{eqn:mv-determine-beta1-2}
    \end{equation}
    by the condition \eqref{condition:clique-end}. Since $\bdbeta_1^1=1$, the moving sequence $\bdlambda_F$ in \eqref{eqn:moving-seq} completely determines the sequence $\bdbeta_1$, and consequently, determines the maximal clique $F$ itself. In short, knowing the maximal clique $F$ is equivalent to knowing the moving sequence $\bdlambda_F$.
\end{Remark}

\subsection{Complete intersection quotients}
\label{ss:linear-quotients}
Next, we consider the following total order $\succ$ on the minimal monomial generating set of $(\ini(P))^{\vee}$:
\[
    \widehat{\bdY^{F_1}}\succ \widehat{\bdY^{F_2}} \qquad \Leftrightarrow \qquad \bdY^{F_1}>_{\lex} \bdY^{F_2}.
\]
We will also write it correspondingly as $F_1\succ F_2$. 
The very first maximal clique of $\calG$ will be denoted sometimes by \emph{$F_{-1}$}. As a reminder, we have two different orderings on the variables. Consequently, we have two different strategies for this ordering. They are tailored to the different cases studied later.

In this subsection, the plan is to show, for each maximal clique $F$ as in \eqref{eqn:F-beta}, the colon ideal
\begin{equation}
    I_F\coloneqq \braket{\widehat{\bdY^{F'}}\mid \text{$F'$ is also a maximal clique of $\calG$ such that $F'\succ F$}}:\widehat{\bdY^{F}}
    \label{eqn:colon-ideal}
\end{equation}
is a complete intersection.
Again, we will only consider the case when $c> r+d$; the case when $r\le c \le r+d$ will be left in the final subsection.
Now, take arbitrary $\bdY^{H}$ in the minimal monomial generating set of the colon ideal $I_F$. Therefore, there exists some $F'\succ F$ such that
\[
    \braket{\widehat{\bdY^{F'}}}:\widehat{\bdY^{F}}=\braket{\bdY^{H}}.
\]
Whence, $H=F\setminus F'$. Suppose for later reference that
\begin{equation}
    F'=\Set{Y_{\bdalpha_1}>\cdots>Y_{\bdalpha_N}}.
    \label{eqn:F-alpha}
\end{equation}
The minimal monomial generating set of $I_F$ consists of two types of squarefree monomials, and we will discuss them in detail.

\begin{Observation} 
    [\textbf{Tail generators}]
    \label{case:tail}
    Suppose that $\bdalpha_1\ne \bdbeta_1$. Since $\bdY^{F'}\succ \bdY^{F}$, $Y_{\bdalpha_1}>Y_{\bdbeta_1}$ and consequently
    $Y_{\bdbeta_1}$ is not the very first ring variable  $Y_{1,2+d,3+2d,\dots,r+(r-1)d}$ of $\KK[\bdY]$.

    For each $j=1,2,\dots,r-1$, we introduce the invariant
    \begin{equation}
        \delta_j\coloneqq \min\Set{i| \bdbeta_i^{\{j,...,r-1\}}=\bdbeta_N^{\{j,...,r-1\}}}
        \label{eqn:delta}
    \end{equation}
    for partially measuring how far the variables in $F$ are from the last one.
    It is clear that
    \[
        \delta_1\ge \delta_2\ge \cdots \ge \delta_{r-1}>1
    \]
    from the definition and the condition \eqref{condition:clique-end}.
    Let $k\le r-1$ be the least positive integer such that
    \begin{equation}
        \underline{1<_d 2+d <_d  3+2d \cdots <_d k+(k-1)d} <_d 
        \underline{\bdbeta_{\delta_k-1}^k <_d \bdbeta_{\delta_k-1}^{k+1} <_d \cdots <_d \bdbeta_{\delta_k-1}^{r-1}}.
        \label{sequence-gamma}
    \end{equation}
    Notice that $\bdbeta_{\delta_{r-1}-1}^{r-1}+1=\bdbeta_{\delta_{r-1}}^{r-1}=\bdbeta_N^{r-1}=\bdbeta_1^r$ and $Y_{\bdbeta_1}$ is not the first variable of the ring. Thus, $\bdbeta_{\delta_{r-1}-1}^{r-1}\ge \bdbeta_1^r-1\ge r+(r-1)d$. Furthermore, note that the two underlined parts hold automatically. Thus,
    the condition in \eqref{sequence-gamma} always holds for $k=r-1$ and the above least $k\le r-1$ is well-defined. Furthermore, there is actually only one inequality to check:
    \begin{equation}
        k+(k-1)d <_d \bdbeta_{\delta_k-1}^k, \quad \text{i.e.,}\quad k+kd < \bdbeta_{\delta_k-1}^k.
        \label{eqn:delta-k}
    \end{equation}
    For each $j\le r-1$, notice that 
    \[
        \bdbeta_{\delta_j-1}^j\ge \bdbeta_{\delta_j}^j-1=\bdbeta_N^j-1=\bdbeta_1^{j+1}-1\ge j+1+(j+1-1)d-1=j+jd.
    \]
    Thus, by the minimality of the chosen $k$, one actually has
    \begin{equation}
        \bdbeta_1^j=j+(j-1)d \qquad \text{ for $j=1,2,\dots,k$.}
        \label{eqn:1-k-minimal}
    \end{equation}

    Now, for the fixed $k$ introduced above, we call the sequence in \eqref{sequence-gamma} as $\bdgamma$. It is clear that $\bdbeta_1^{\le k}=\bdgamma^{\le k}$ by \eqref{eqn:1-k-minimal}. Meanwhile, for $j=k+1,\dots,r$, one has
    \[
        \bdgamma^j=\bdbeta_{\delta_k-1}^{j-1}\le \bdbeta_N^{j-1}=\bdbeta_1^j.
    \]
    Thus, $\bdbeta_1 \ge \bdgamma$ componentwise. Meanwhile, by the choice of $\delta_k$, $\bdbeta_{\delta_k-1}^{\{k,...,r-1\}}\ne\bdbeta_N^{\{k,...,r-1\}}$. Consequently $Y_{\bdbeta_1}<Y_{\bdgamma}$ in both cases in \eqref{2-orderings}. Now, one can easily construct some maximal clique $F''$ of $\calG$ as
    \begin{align*}
        F''=\bigg\{Y_{\bdgamma_1=\bdgamma}>Y_{\bdgamma_2}>\cdots>Y_{\bdgamma_{k'-1}}>Y_{\bdgamma_{k'}=\bdbeta_1}>Y_{\bdgamma_{k'+1}=\bdbeta_2}>\cdots>
        \\
        Y_{\bdgamma_{k'+\delta_k-2}=\bdbeta_{\delta_k-1}} >Y_{\bdgamma_{k'+\delta_k-1}}>\cdots>Y_{\bdgamma_N} \bigg\},
    \end{align*}
    using \Cref{MaxCli}.
    Here $k'-1=\sum_{j=1}^{r}\bdbeta_1^j-\sum_{j=1}^{r}\bdgamma^j=\sum_{j=k+1}^{r}(\bdbeta_1^j-\bdgamma^j)$. And it is clear that
    \[
        F\setminus F''=\Set{Y_{\bdbeta_{\delta_k}}>\cdots >Y_{\bdbeta_N}}.
    \]

    Now, back to the maximal clique $F'\succ F$ with $\bdalpha_1\ne \bdbeta_1$. We claim that $F\setminus F'' \subseteq F\setminus F'$. Otherwise, there exists some $i$ with $\delta_k\le i\le N$ such that $Y_{\bdbeta_i}\in F'$. Notice that $\bdbeta_1^{\le k}\le \bdalpha_1^{\le k}$ componentwise by the condition \eqref{eqn:1-k-minimal}. Meanwhile, 
    \[
        \bdbeta_1^{\ge k+1}=\bdbeta_N^{\{k,...,r-1\}}=\bdbeta_i^{\{k,...,r-1\}}\stackrel{\text{comp.}}{\le} \bdalpha_N^{\{k,...,r-1\}}=\bdalpha_1^{\ge k+1}.
    \]
    Here, the equalities are due to the choice of $\delta_k$ and the condition \eqref{condition:clique-end}, while the componentwise inequality is due to the condition \eqref{eqn:clique-condition} and the fact that $Y_{\bdbeta_i}\in F'$. Thus, we have $\bdalpha_1\ge \bdbeta_1$ componentwise. Consequently, $Y_{\bdalpha_1}\le Y_{\bdbeta_1}$ in term of the lexicographical order that we define in \eqref{2-orderings}, which is a contradiction. This confirms our claim.

    Now, by the minimality of $H$, we must have $H=\{Y_{\bdbeta_{\delta_k}}>\cdots>Y_{\bdbeta_N}\}$. We will call the monomial $\bdY^H$ as the \emph{tail generator} of the corresponding colon ideal $I_F$. 
\end{Observation}

\begin{Observation}
    [\textbf{Corner generators}]
    \label{case:corner}

    Suppose that $\bdalpha_1=\bdbeta_1$, whence $\bdalpha_N=\bdbeta_N$ by the condition \eqref{condition:clique-end}. 
    Furthermore, if $\bdlambda_F$ and $\bdlambda_{F'}$ are the corresponding moving sequences of $F$ and $F'$ respectively, then $\bdlambda_F$ and $\bdlambda_{F'}$ coincide if we deem them as multi-sets by the conditions \eqref{eqn:mv-determine-beta1-1} and \eqref{eqn:mv-determine-beta1-2}.

    Depending on the linear order of the variables given at \eqref{2-orderings}, we have two cases.

    \begin{enumerate}[i]
        \item \textbf{The ordering is of $lex$ type.}
            In this case, if we also consider $\bdlambda_F$ and $\bdlambda_{F'}$ as tuples in $\ZZ^{r-1}$, then $F'\succ F$ if and only if
            \begin{equation}
                \text{the leftmost nonzero component of the difference $\bdlambda_{F'}-\bdlambda_{F}$ is positive.}
                \label{characterization-1}
            \end{equation}
            Now consider the set
            \begin{equation}
                K_F\coloneqq \{1\le k\le N-2 \mid  j_k<j_{k+1}\}
                \label{eqn:K_F}
            \end{equation}
            where the $j_k$'s are given in \eqref{condition:move-1}.
            For each fixed $k\in K_F$, we then have some maximal clique $F''$ with the moving sequence
            \[
                \bdlambda_{F''}=(j_1,\dots,j_{k-1},j_{k+1},j_k,j_{k+2},\dots,j_{N-1}).
            \]
            Here, one uses condition \eqref{eqn:clique-condition} to verify the existence (compatibility) of $F''$.  It is clear that $F''\succ F$ and $F\setminus F''=\{Y_{\bdbeta_{k+1}}\}$.

            It remains to show that a minimal monomial generator of the colon ideal $I_F$ comes from this situation if and only if it belongs to
            \begin{equation}
                \Set{Y_{\bdbeta_{k+1}}\mid k\in K_F}.
                \label{eqn:corner-type-generators}
            \end{equation}
            Notice that the ``if'' part has been shown by the above argument. And the indices in $K_F$ cut $\bdlambda_F$ into a collection of weakly decreasing subsequences.

            As a quick example, when $(r,c,d)=(3,6,1)$, we have a maximal clique
            \[
                F=\Set{Y_{1,3,5}\xrightarrow{j_1} Y_{1,3,6}\xrightarrow{j_2} Y_{1,3,7}\xrightarrow{j_3} Y_{1,4,7}\xrightarrow{j_4} Y_{1,4,8}\xrightarrow{j_5} Y_{1,5,8}\xrightarrow{j_6} Y_{2,5,8}\xrightarrow{j_7} Y_{3,5,8}}.
            \]
            The moving sequence is 
            \[
                \bdlambda_F=(j_1,j_2,\dots,j_7)=(\underline{3,3,2},\underline{3, 2, 1, 1}).
            \]
            Whence, $K_F=\{3\}$ and the expected generator in this situation is solely $Y_{1,4,7}$. And the index $3$ in $K_F$ cuts $\bdlambda_F$ into the underlined format.

            Now, back to the argument for $F'$ in \eqref{eqn:F-alpha} such that the corresponding monomial $\bdY^H$ is a minimal monomial generator of the colon ideal $I_F$. 
            Let $A=\Set{1\le k\le N-1\mid \bdalpha_{k+1}\ne \bdbeta_{k+1}}$. 
            It remains to show that $A\cap K_F\ne \varnothing$ due to the expected minimal generating set \eqref{eqn:corner-type-generators} and the fact that $H=F\setminus F'$.  
            Notice that the maximal cliques $F$ and $F'$ surely contain
            $\Set{Y_{\bdbeta_{k+1}}\mid k\notin A}\cup\Set{Y_{\bdbeta_1},Y_{\bdbeta_N}}$.
            Now, suppose for contradiction that $A\cap K_F= \varnothing$. Consequently, both $F$ and $F'$ contain
            \begin{equation*}
                \Set{Y_{\bdbeta_{k+1}}\mid k\in K_F}\cup\Set{Y_{\bdbeta_1},Y_{\bdbeta_N}}=\Set{Y_{\bdbeta_{j_1}}>Y_{\bdbeta_{j_2}}>\cdots>Y_{\bdbeta_{j_t}}}.
            \end{equation*}
            It is clear that $j_1=1$ and $j_t=N$. Furthermore, for each $s\in [t]$, we have $\bdbeta_{j_s}=\bdalpha_{j_s}$. For each $s\in [t-1]$, suppose that $\bdlambda_s$ is the sequence of movement in $F$ from $\bdbeta_{j_{s}}$ to $\bdbeta_{j_{s+1}}$ and similarly introduce $\bdlambda_s'$ for $F'$. It is clear that $\bdlambda_s$ is a weakly decreasing sequence, and 
            \[
                \left|\Set{\text{$r'$-movements in $\bdlambda_s$}}\right|= \bdbeta_{j_{s+1}}^{r'}-\bdbeta_{j_{s}}^{r'}=
                \left|\Set{\text{$r'$-movements in $\bdlambda_s'$}}\right|
            \]
            for every $r'\in [r]$. 
            Thus, $F\succ F'$ by the characterization in \eqref{characterization-1}, 
            which contradicts the assumption that $F'\succ F$. 

            In short, in this $lex$ type ordering case with the assumption that $\bdalpha_1=\bdbeta_1$, the monomial $\bdY^H$ is a minimal generator of $I_F$ if and only if this monomial belongs to the set in \eqref{eqn:corner-type-generators}.

        \item \textbf{The ordering is of $revlex$ type.}
            In this case, if we also consider $\bdlambda_F$ and $\bdlambda_{F'}$ as tuples in $\ZZ^{r-1}$, then $F'\succ F$ if and only if the leftmost nonzero component of the difference $\bdlambda_{F'}-\bdlambda_{F}$ is \emph{negative}.

            Now, let $K_F'$ be the set of all indices $k$ such that $j_k>j_{k+1}$ with the \emph{additional requirement} that we have some maximal clique $F''$ with
            \[
                \bdlambda_{F''}=(j_1,\dots,j_{k-1},j_{k+1},j_k,j_{k+2},\dots,j_{N-1}).
            \]
            Notice that, unlike the $lex$ type case, such a maximal clique $F''$ does not exist by default. Now, for any such $k$ in $K_F'$ with the accompanied $F''$, it is clear that $F''\succ F$ and $F\setminus F''=\{Y_{\bdbeta_{k+1}}\}$. One can argue as in the $lex$ type case that the minimal monomial generators coming from this situation are purely
            \begin{equation}
                \Set{Y_{\bdbeta_{k+1}}\mid k\in K_F'}.
                \label{eqn:corner-type-generators-revlex}
            \end{equation}
    \end{enumerate}

    Two maximal cliques in the case $(r,c,d)=(2,8,1)$ are pictured in \Cref{fig:dim2}. 
    \begin{figure}[htb]
        \scalebox{0.6}{
        \begin{tikzpicture}[scale=1.3]
        \draw[ultra thick] (1,4)--(1,5)--(1,6)--(2,6)--(2,7)--(3,7)--(3,8)--(3,9)--(4,9);
        \node[below right] at (1,4) {$Y_{1,4}$};
        \node[below right] at (1,5) {$Y_{1,5}$};
        \node[below right] at (1,6) {$Y_{1,6}$};
        \node[below right] at (2,6) {$Y_{2,6}{}^\ast$};
        \node[below right] at (2,7) {$Y_{2,7}$};
        \node[below right] at (3,7) {$Y_{3,7}{}^\ast$};
        \node[below right] at (3,8) {$Y_{3,8}$};
        \node[below right] at (3,9) {$Y_{3,9}$};
        \node[below right] at (4,9) {$Y_{4,9}{}^\vartriangle$};
        \draw[step=1,gray, very thin] (0.5,3.5) grid (4.5,9.5);
        \end{tikzpicture}}
        \qquad \qquad
        \scalebox{0.6}{
        \begin{tikzpicture}[scale=1.3]
        \draw[ultra thick] (1,4)--(1,5)--(1,6)--(2,6)--(2,7)--(3,7)--(3,8)--(4,8)--(4,9);
        \node[below right] at (1,4) {$Y_{1,4}$};
        \node[below right] at (1,5) {$Y_{1,5}$};
        \node[below right] at (1,6) {$Y_{1,6}$};
        \node[below right] at (2,6) {$Y_{2,6}{}^\ast$};
        \node[below right] at (2,7) {$Y_{2,7}$};
        \node[below right] at (3,7) {$Y_{3,7}{}^\ast$};
        \node[below right] at (3,8) {$Y_{3,8}$};
        \node[below right] at (4,8) {$Y_{4,8}{}^{\ast,\vartriangle}$};
        \node[below right] at (4,9) {$Y_{4,9}{}^\vartriangle$};
        \draw[step=1,gray, very thin] (0.5,3.5) grid (4.5,9.5);

        \end{tikzpicture}}
        \caption{Two maximal cliques for $(r,c,d)=(2,8,1)$ in the $lex$ type ordering. Variables in the respective tail generators are marked with $\vartriangle$ and corner generators are marked with $\ast$.} 
        \label{fig:dim2}
    \end{figure}
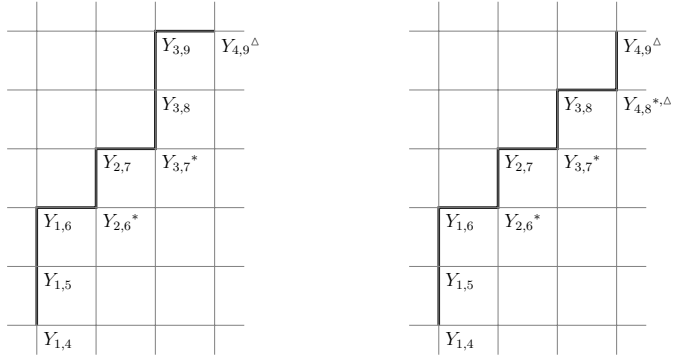
    The linear generators generated by the current method are marked with $\ast$.
    For the obvious pictorial reasons, we will call the minimal generators given by either \eqref{eqn:corner-type-generators} or \eqref{eqn:corner-type-generators-revlex}, depending on the ordering of the variables, as the \emph{corner generators} of the corresponding colon ideal $I_F$.
\end{Observation}

\begin{Remark}
    \label{rmk:tail-corner}
    As a reminder, the tail generator may or may not belong to the minimal monomial generating set of $I_F$, depending on whether it will be canceled out by some corner generator. 
    For instance, when $(r,c,d)=(2,8,1)$  we will have a maximal clique
    \[
        F=\{Y_{1,4}>Y_{1,5}>Y_{1,6}>Y_{2,6}>Y_{2,7}>Y_{3,7}>Y_{3,8}>Y_{4,8}>Y_{4,9}\};
    \]
    see also the right-hand side picture in \Cref{fig:dim2}.
    Using the $lex$ type ordering in \eqref{2-orderings}, we will have a tail generator 
    \[
        \bdY^H=Y_{4,8}Y_{4,9}
    \]
    regarding this $F$. But it is not difficult to see that the set of corner generators is precisely $\{Y_{2,6},Y_{3,7},Y_{4,8}\}$. Whence, the tail generator is canceled out and the colon ideal $I_F=\braket{Y_{2,6},Y_{3,7},Y_{4,8}}$.
    
    On the other hand, all the corner generators will stay for sure, since they are linear. Whence, the colon ideal $I_F$ is always a monomial complete intersection ideal.
\end{Remark}

\subsection{Upper bound for the projective dimension}
\label{ss:max-length}

Recall that for a homogeneous ideal $I$ and a homogeneous element $f$, the short exact sequence
\[
    0\to R/(I:f) \to R/I \to R/(I+f) \to 0
\]
will induce the well-known fact
\[
    \projdim(R/(I+f))\le \max\{\projdim(R/(I:f))+1,\projdim(R/I)\}.
\]
Whence, if $I=\braket{f_1,\dots,f_s}$ such that for each $t=2,3,\dots,s$, the colon ideal $\braket{f_1,\dots,f_{t-1}}:f_t$ is a complete intersection of codimension $\ell_t$, then
\begin{equation}
    \projdim(I)=\projdim(R/I)-1\le \max_{2\le t\le s} \ell_t.
    \label{eqn:general-upper-bound}
\end{equation}

By the argument in the previous subsection, we see that the minimal monomial generating set of the colon ideal $I_F$ corresponding to the maximal clique $F$ is given by
\begin{enumerate}[a]
    \item some corner generators, which are all linear, and
    \item a unique tail generator, if it exists and is not canceled out by corner generators.
\end{enumerate}
Furthermore, notice that if the starting variable of $F$ is the first variable of the ring, the tail generator does not exist. 

In this subsection, we intend to give an upper bound of the projective dimension of $(\ini(P))^\vee$. Since the tail generator, if exists, only contributes $1$ to the codimension of the complete intersection ideal $I_F$, roughly speaking, it suffices to figure out when the maximal number of the corner generators is achieved.

\begin{Observation}
    \label{lambdaIncrease}
    In the following, we assume that the maximal number of the corners is achieved for some maximal clique $F=\Set{Y_{\bdbeta_1}>\cdots>Y_{\bdbeta_N}}$. We will only explain in the case of $lex$ type ordering of variables, which is sufficient for obtaining the upper bound in \eqref{eqn:max-corner}.
    Notice that the moving sequence $\bdlambda_{F}$ is a concatenation of some strictly increasing subsequences:
    \[
        \bdlambda_{F}= (\bdlambda_1,\bdlambda_2,\dots,\bdlambda_s).
    \]
    Here, we always require that the strictly increasing subsequences above to be maximal with respect to their individual lengths. For instance, for the example in \Cref{rmk:tail-corner}, the moving sequence is
    \[
        \bdlambda_F=(\underline{3},\underline{3},\underline{2},\underline{2},\underline{2,3},\underline{1},\underline{1}),
    \]
    where the maximally strictly increasing subsequences are underlined. In this special case, only $\bdlambda_5=(2,3)$ has length $2$; all other subsequences have only length $1$.

    Now, back to our moving sequence $\bdlambda_F$, we observe that the two sets
    \begin{gather*}
        \Set{2\le i\le N-1\mid i-1\notin K_F}
        \intertext{and}
        \Set{\text{the position of the variable in $F$ after the movements of $\bdlambda_j$}|j=1,2,\dots,s-1}
    \end{gather*}
    are identical, where the set $K_F$ is defined in \eqref{eqn:K_F}.
    Therefore, maximizing the cardinality of $K_F$ is equivalent to minimizing the number of strictly increasing subsequences.
\end{Observation}

Consequently, we obtain the following main result of this subsection.     

\begin{Proposition}
    Under the $lex$ type or $revlex$ type ordering of the variables, let $\ell_F$ be the codimension of the complete intersection colon ideal $I_F$ corresponding to the maximal clique $F$ studied above. Then,
    \begin{equation}
        \projdim((\ini(P))^\vee)\le \max_{F\ne F_{-1}}\ell_F,
        \label{eqn:upper-bound}
    \end{equation}
    where $F_{-1}$ is the first maximal clique of $\calG$.
    In particular, when $c> r+d$, we have
    \[
        \projdim((\ini(P))^\vee)\le (N-1)-\ceil{(N-1)/r}+1.
    \]
\end{Proposition}

\begin{proof}
    The first inequality is from \eqref{eqn:general-upper-bound}.   With \Cref{lambdaIncrease}, we notice that each strictly increasing subsequence has a length at most $r$. Therefore,
    \begin{equation}
        |K_F|\le N-1-\ceil{(N-1)/r},
        \label{eqn:max-corner}
    \end{equation}
    since $N-1=|F|-1$ is the length of $\bdlambda_F$. 
    Note that the final $+1$ of the ``in particular" part comes from the potential contribution of the tail generator.
\end{proof} 

\begin{Remark}
    When applying the $revlex$ type of ordering, by symmetry, we will break the moving sequence $\bdlambda_F$ into maximally strictly decreasing subsequence, again, written as $\bdlambda_{F}= (\bdlambda_1,\bdlambda_2,\dots,\bdlambda_s)$. The final movement of each $\bdlambda_j$ will not contribute to any corner generator. Furthermore, some additional movement may fail to contribute, due to the modified description of $K_F'$ in \Cref{case:corner}. For instance, when $(r,c,d)=(3,6,1)$, we have a maximal clique
    \[
        F=\Set{Y_{1,3,5}>Y_{1,3,6}>Y_{1,4,6}>Y_{2,4,6}>Y_{2,4,7}>Y_{2,5,7}>Y_{2,5,8}>Y_{3,5,8}}.
    \]
    The associated moving sequence is $\bdlambda_F=(\underline{3, 2, 1}, \underline{3, 2}, \underline{3, 1})$. However, $Y_{2,5,8}$ is the only corner generator.
\end{Remark}

Actually, we seek more than a mere upper bound. To get over this discrepancy, we need extra tools.
Recall the following key facts of \emph{iterated mapping cones} from \cite[Construction 27.3]{MR2560561}. Let $I$ be an ideal minimally generated by monomials $m_1,\dots,m_r$ in some polynomial ring $S$. For $1\le i\le r$, write $I_i\coloneqq \braket{m_1,\dots,m_i}$. Then we have a short exact sequence of homogeneous modules
\[
    0\to (S/(I_i:m_{i+1}))(m_{i+1})\stackrel{m_{i+1}}\to S/I_i\to S/I_{i+1}\to 0,
\]
where the comparison map is the multiplication by $m_{i+1}$. Here, $(S/(I_i:m_{i+1}))(m_{i+1})$ is obtained by shifting the module $S/(I_i:m_{i+1})$ in multidegree by $m_{i+1}$ so that the comparison map is of degree $0$.
Assume that a multigraded free resolution $\bfF_i$ of $S/I_i$ is already known and that a multigraded free resolution $\bfG_i$ of $S/(I_i:m_{i+1})$ is also known, then one can construct the mapping cone and obtain a multigraded free resolution $\bfF_{i+1}$ of $S/I_{i+1}$.

Now, let us focus on the complete intersection quotients studied above in the polynomial ring $S=\KK[\bdY]$.
For each given maximal clique $F$ which is different from the very first one $F_{-1}$, the colon ideal is a complete intersection ideal $I_F$. Let $G_F\coloneqq G(I_F)$ be the minimal monomial generating set of the monomial ideal $I_F$. Whence, it can be resolved minimally by the Koszul complex, whose top module has rank $1$ at the homological degree $\ell_F=|G_F|$, and is shifted in multidegree by $m_F\coloneqq \prod_{f\in G_F} f$. This module is denoted by $S(m_F)$ in \cite{MR2560561}. Similarly, the free module at the homological degree $\ell_F-1$ of this Koszul complex is $\bigoplus_{f\in G_F}S(m_F/f)$.

\begin{Definition}
    For each maximal clique $F$ studied above, we define the \emph{essential part} of $F$ as
    \[
        \Ess(F)\coloneqq \Set{g\in F| \text{for each $f\in G_F$, $g\dividesnot f$}}=F\setminus \Supp(m_F),
    \]
    where $\Supp(m_F)$ denotes the set of variables that divide the monomial $m_F$.
\end{Definition}

Since the monomial generator of $(\ini(P))^\vee$ that we are dealing with is $\widehat{\bdY^F}$,
\emph{in the mapping cone}, we have to further shift in multidegree by $\widehat{\bdY^F}$. Consequently, we will have a rank $1$ module shifted in multidegree by $m_F \widehat{\bdY^F}=\widehat{\bdY^{\Ess(F)}}$ at the homological degree $\ell_F+1$. Meanwhile, the free module at the homological degree $\ell_F$ of the shifted Koszul complex is
\[
    \bigoplus_{f\in G_F}S(m_F \widehat{\bdY^F}/f)=\bigoplus_{f\in G_F}S(\widehat{\bdY^{\Ess(F)\,\sqcup\,\Supp(f)}}).
\]

Note that the free resolution given by the iterated mapping cone may not necessarily be minimal. 
For instance, when $(r,c,d)=(3,6,1)$ and if we apply the $lex$ type ordering of the variables, then 
\[
    \pd((\ini(P))^\vee)=4<\max_{F\ne F_{-1}}\ell_F=5.
\]
This is also the reason why we need to consider the $revlex$ type of ordering in some cases.

We need the following key observation to determine the regularity of the fiber cone.

\begin{Lemma}
     \label{F_0}
   Suppose that there is some maximal clique $F_0$ satisfying the following three conditions:
    \begin{enumerate}[label=\textup{(C\arabic*)}]
        \item \label{C1} 
            $\ell_{F_0}=\max_{F\ne F_{-1}}\ell_F$,
        \item \label{C2} 
            there is no other maximal clique $F'$ with $\Ess(F')=\Ess(F_0)$, and
        \item \label{C3} 
            there is no maximal clique $F'$ and $f\in G_{F'}$ with the disjoint union $\Ess(F')\sqcup\Supp(f)=\Ess(F_0)$.
    \end{enumerate}
    Then, the projective dimension is given by 
    \[
        \projdim((\ini(P))^\vee)= \max_{F\ne F_{-1}}\ell_F.
    \]
\end{Lemma}

\begin{proof}
    The condition \ref{C1} is obviously necessary in view of \eqref{eqn:upper-bound}. Now, we go through the iterated mapping cone process.
    
    When doing the mapping cone at the stage of the maximal clique $F=F_0$, the top rank $1$ module is now at the homological degree $\ell_{F_0}+1$. It does not contribute to the projective dimension as expected, precisely when it actually cancels some free module of rank $1$ at the homological degree $\ell_{F_0}$, shifted at the same multidegree. By the maximality of $\ell_{F_0}$, the latter free module has two sources. The first source is the top free module of some Koszul complex added earlier by some maximal clique $F'$ with $\ell_{F'}=\ell_{F_0}-1$. Since we will have $\ell_{F_0}\ge 2$ in later discussions, we may assume that $F'\ne F_{-1}$. This leads to the introduction of the condition \ref{C2}. The second source is the free module from the second highest homological degree in earlier Koszul complexes. This leads to the introduction of the condition \ref{C3}.
    
    If the top rank $1$ module ``survives'' at this stage, it will ``survive forever''. This is because it stays at the homological degree $\ell_{F_0}+1$, while all existing free modules that are canceled out in later mapping cones, undoubtedly come from the homological degrees at most $\ell_{F_0}$.
\end{proof}
   
\begin{Remark}
    \label{rmk:F_0}
    \begin{enumerate}
        \item  Suppose that we have a maximal clique $F'$ and $f\in G_{F'}$ with the disjoint union $\Ess(F')\sqcup\Supp(f)=\Ess(F_0)$. Since the starting variable of $F_0$ belongs to $\Ess(F_0)$ and the starting variable of $F'$ belongs $\Ess(F')$, these two starting variables must coincide. Furthermore, $f$ cannot be this common variable.
        \item Similarly, if we have two maximal cliques $F'$ and $F_0$ with $\Ess(F')=\Ess(F_0)$, then the starting variable of $F'$ is identical to that of $F_0$.
    \end{enumerate}
\end{Remark}

\subsection{The $2r+d\le c$ case}
\label{subsection:c>=2r+d}
Starting from this subsection, we will pin down the concrete projective dimension of $(\ini(P))^\vee$. For that purpose, suppose that $N-1= pr+q$ such that $1\le q\le r$. 
To achieve the maximum in \eqref{eqn:max-corner}, the most natural strategy in mind is to assume that $\bdlambda_{F}$ is the concatenation of $p$ subsequences of the form $(1,2,\dots,r)$ first, and then followed by $1$ subsequence of the form $(1,2,\dots,q)$. 

We have seen earlier that the moving sequence $\bdlambda_F$ in \eqref{eqn:moving-seq} actually determines the maximal clique $F$. Using the above strategy, it is not difficult to see that  the first variable of $F$ has the form
\[
    \bdbeta_1=(1,1+(p+1),1+2(p+1),\dots,1+q(p+1),1+q(p+1)+p,\dots,1+q(p+1)+(r-q-1)p)
\]
by \eqref{eqn:mv-determine-beta1-1} and \eqref{eqn:mv-determine-beta1-2}.
Since we have the $<_d$ compatibility requirement, to make the moving sequence legal, we have the following considerations.
\begin{enumerate}[i]
    \item If $ r-q-1\ge 1$, i.e., if $q\le r-2$, then the compatibility requirements are reduced to the single inequality $1+q(p+1)+1<_d 1+q(p+1)+p$, which is equivalent to saying $c\ge 2r+d+2$.
    \item If $q=r-1$, then the compatibility requirements are reduced to the single inequality $1+1<_d 1+(p+1)$, which is equivalent to saying $c\ge 2r+d$.
    \item If $q=r$, then the compatibility requirements are reduced again to the single inequality $1+1<_d 1+(p+1)$, which is equivalent to saying $c\ge 2r+d+1$.
\end{enumerate}

Notice that if $c=2r+d$, then $N-1=(d+1)r+(r-1)$. Whence, $q=r-1$. And if $c=2r+d+1$, then $N-1=(d+1)r+r$. Whence, $q=r$. Now, it suffices to consider the following two cases.

\begin{Observations}
    \label{rmk:upper-bound}
    \begin{enumerate}[a]
        \item If $c\ge 2r+d$ and $q<r$, then the above argument shows that we can find suitable maximal clique $F_0$, whose moving sequence $\bdlambda_{F_0}$ has $p+1$ strictly increasing subsequences, and the final movement in the final subsequence is $q$. Since $1+(p+1)>2+d$ in these cases, the special $k$ in \eqref{sequence-gamma} is actually $1$, by \eqref{eqn:1-k-minimal}. Now, since $q<r$, we have a linear tail generator by the description in subsection \ref{ss:linear-quotients}. This generator surely cannot be canceled by a corner generator. Hence the maximum at the right-hand side of \eqref{eqn:upper-bound} is exactly 
        \[
        (N-1)-\ceil{(N-1)/r}+1=(N-1)-\floor{(N-1)/r}.
        \]

        \item If $c\ge 2r+d+1$ and $q=r$,  suppose that the moving sequence $\bdlambda_{F_0}$ of some suitable maximal clique $F_0$ has the least number of strictly increasing subsequences. Since $q=r$, this forces $\bdlambda_{F_0}$ to be exactly the concatenation of $p+1$ subsequences of the form $(1,2,\dots,r)$. Whence, the special $k$ can be seen to still be $1$, while the tail generator has degree $2$. This tail generator will be canceled by the final corner generator $Y_{\bdbeta_{N-1}}$. In other words, the minimal monomial generating set of $I_{F_0}$ consists solely of corner generators. Since $r$ divides $N-1$, the maximum at the right-hand side of \eqref{eqn:upper-bound} is then $(N-1)-(N-1)/r$.
    \end{enumerate}
\end{Observations}

Here is the main result of this subsection.

\begin{Proposition}
    \label{thm-main-result}
    If $2r+d \leq c$, then $\projdim((\ini(P))^\vee)=N-1-\floor{(N-1)/r}$.
\end{Proposition}  

\begin{proof}
    Let $F_0$ be the maximal clique given in \Cref{rmk:upper-bound}. We will verify that $F_0$ satisfies the conditions \ref{C2} and \ref{C3} in \Cref{F_0}, where \ref{C1} holds already.

    Suppose that $F'$ is also a maximal clique with $\Ess(F_0)=\Ess(F')$. Without loss of generality, we may assume that $F_0=\Set{Y_{\bdbeta_1}>\cdots>Y_{\bdbeta_N}}$ while $F'=\Set{Y_{\bdalpha_1}>\cdots>Y_{\bdalpha_N}}$.
    \begin{enumerate}[i]
        \item 
        
        We have seen in \Cref{rmk:F_0} that $\bdbeta_1=\bdalpha_1$.
        \item \label{2r+d:ii}
            Notice that $Y_{\bdbeta_{1+r}},Y_{\bdbeta_{1+2r}},\dots,Y_{\bdbeta_{1+pr}}$ are the next $p$ elements in order in $\Ess(F_0)$. Therefore, $Y_{\bdalpha_{1+s_1r+s_2}}$ are all corner generators of $I_{F'}$ for $0\le s_1\le p-1$ and $1\le s_2 \le r-1$. But in any legitimate moving sequence, any strictly increasing subsequence has length at most $r$. Hence, the movement from $Y_{\bdalpha_1}$ to $Y_{\bdalpha_{1+pr}}$ in $F'$ consists precisely of the concatenation of $p$ subsequences of the form $(1,2,\dots,r)$, just as in $F$. Thus, as $\bdbeta_1=\bdalpha_1$, it forces $\bdalpha_j=\bdbeta_j$ for $j=1,2,\dots,1+pr$.
        \item For the final moving subsequence, we have two cases. Note that since $\bdalpha_1=\bdbeta_1$, the special $k$ in \eqref{sequence-gamma} for $F'$ is actually $1$ as well, by \eqref{eqn:1-k-minimal}. 
            \begin{enumerate}[a]
                \item Suppose that $q<r$. Whence, by the discussion in \Cref{rmk:upper-bound}, we have no further variable for $\Ess(F_0)$.  
               
                Since $\bdalpha_N=\bdbeta_N$, the remaining moving subsequence of $F'$ is just a rearrangement of $(1,2,3,\dots,q)$ while $q<r$. Whence, the tail generator corresponding to $F'$ is linear as well. In turn, $Y_{\bdalpha_{1+pr+1}},\dots,Y_{\bdalpha_{N-1}}$ are all corner generators for $F'$, since $\Ess(F_0)=\Ess(F')$. Now, the moving subsequence corresponding to the movement from $Y_{\bdalpha_{1+pr}}$ to $Y_{\bdalpha_{N}}$ has to be a strictly increasing one, namely $(1,2,\dots,q)$. In short, $\bdlambda_{F_0}=\bdlambda_{F'}$. Consequently, we have $F_0=F'$, as expected.
                
                \item Suppose that $q=r$. Whence, by the discussion in \Cref{rmk:upper-bound}, $Y_{\bdbeta_N}$ is the unique additional variable for $\Ess(F_0)$. Since $\Ess(F_0)=\Ess(F')$ while the final variable of $F'$ can never be a corner generator, this final variable will belong to a tail generator of degree at least $2$. On the other hand, since the remaining moving subsequence contains exactly one copy of $r$, this tail generator has degree at most $2$. This implies that $Y_{\bdalpha_{1+pr+1}},\dots,Y_{\bdalpha_{N-1}}$ are all corner generators regarding $F'$, and $Y_ {\bdalpha_{N-1}}$ will cancel the tail generator of degree $2$. As in the previous case, we will consequently have $F_0=F'$, as expected.
            \end{enumerate}
    \end{enumerate}

    So far, the conditions \ref{C1} and \ref{C2} have been verified. We continue to verify the condition \ref{C3} in these cases. Assume to the contrary that there is some maximal clique $F'$ and some generator $f\in G_{F'}$ with the disjoint union $\Ess(F')\sqcup\Supp(f)=\Ess(F_0)$. By \Cref{rmk:F_0}, $F_0$ and $F'$ have the same starting variable, which cannot be $f$.
    \begin{enumerate}[a]
        \item When $c\ge 2r+d$ and $q<r$, we have seen that
            \[
                \Ess(F_0)=\Set{Y_{\bdbeta_1}, Y_{\bdbeta_{1+r}},Y_{\bdbeta_{1+2r}},\dots,Y_{\bdbeta_{1+pr}}}.
            \]
            Since $\Supp(f)\subset \Ess(F_0)$ and $f\in G_{F'}$, this $f$ has to be a corner generator in $G_{F'}$ and is linear.

            If $f=Y_{\bdbeta_{1+p'r}}$ with $1\le p'<p$, then in the moving sequence $\bdlambda_{F'}$, there will be a strictly increasing subsequence of length $2r$ moving from $Y_{\bdbeta_{1+(p'-1)r}}$ to $Y_{\bdbeta_{1+(p'+1)r}}$. But the length of a strictly increasing subsequence is at most $r$, and we have a contradiction. 
            
            If $f=Y_{\bdbeta_{1+pr}}$, then either we have a strictly increasing subsequence of length $>r$, or this $f$ belongs to the tail generator. The first case won't happen by the same reason. 
            As for the latter one, notice that the special $k$ in \eqref{sequence-gamma} has to be $1$.
            This implies that starting from $f$ all the movements in $F'$ have to be $r$. 
            But this is impossible, since $\bdbeta_1=\bdalpha_1$ and we have the requirements \eqref{eqn:mv-determine-beta1-1} and \eqref{eqn:mv-determine-beta1-2}.
            
        \item When $c\ge 2r+d+1$ and $q=r$, like above, we have
            \[
                \Ess(F_0)=\Set{Y_{\bdbeta_1}, Y_{\bdbeta_{1+r}},Y_{\bdbeta_{1+2r}},\dots,Y_{\bdbeta_{1+pr}},Y_{\bdbeta_{N=1+(p+1)r}}}.
            \]
            Similarly, we only have to worry about the case when $f$ is the last one, which in this case is $Y_{\bdbeta_N}=Y_{\bdalpha_N}$. Consider the strictly increasing subsequence that moves starting from $Y_{\bdalpha_{1+pr}}=Y_{\bdbeta_{1+pr}}$ regarding $F'$. We claim that the end variable of this round of movements will be the starting variable of the tail generator of $I_{F'}$.
           
            To see it, notice that if the tail generator has a higher degree, it will get canceled out by some corner generator in $G_{F'}$. This will make $Y_{\bdalpha_N}$ appearing in $\Ess(F')$, a contradiction.  On the other hand, if the tail generator has a lower degree, the end variable of this round will appear in $\Ess(F')$, another contradiction.
            
            Notice that the remaining moving subsequence starting from $Y_{\bdalpha_{1+pr}}$ is just a rearrangement of $(1,2,\dots,r)$ by the requirements \eqref{eqn:mv-determine-beta1-1} and \eqref{eqn:mv-determine-beta1-2}. 
            If the tail generator is linear, this generator is precisely $f$. Whence  $\bdlambda_{F'}=\bdlambda_{F_0}$, and consequently $F'=F_0$, a contradiction. 
            On the other hand, when the tail generator is not linear, since $\bdbeta_1=\bdalpha_1$ and we have the special $k$ in \eqref{sequence-gamma} for $F'$ to be $1$, this tail generator has degree $2$ and the final movement is $r$. Therefore, the strictly increasing moving subsequence starting from $Y_{\bdalpha_{1+pr}}$ will be $(1,2,\dots,r-1)$. We again obtain $\bdlambda_{F'}=\bdlambda_{F_0}$, which is a contradiction. And this completes the proof. \qedhere
    \end{enumerate}
\end{proof}

\subsection{The $r+d< c<2r+d$ case}
\label{sec:sharp-bound-inter}

The main result of this subsection is to show the following formula.

\begin{Proposition}
    \label{prop:sharp-bound-inter}
    Suppose that $r+d< c< 2r+d$. Then the projective dimension is given by
    \begin{equation}
        \projdim((\ini(P))^\vee)=dr-2r-3d+2c-2.
        \label{eqn:sharp-bound-inter}
    \end{equation}
\end{Proposition}

Though this is just an intermediate result with respect to our \Cref{thm:main-reg-balanced}, its proof is still quite involved. 
We will denote the upper bound at the right-hand side of \eqref{eqn:upper-bound} by $\bdl_{r,c,d}^{lex}$ or $\bdl_{r,c,d}^{revlex}$, depending on whether we take the $lex$ type or $revlex$ type ordering of variables.  
And here is the strategy for proving \eqref{eqn:sharp-bound-inter}. 
\begin{enumerate}[I]
    \item \label{step-1}
        We first show that the equality \eqref{eqn:sharp-bound-inter} holds when $c=r+d+1$ via the $lex$ type ordering. Whence, the right-hand side of \eqref{eqn:sharp-bound-inter} is given by $\bdl_{r,r+d+1,d}^{lex}$, which surely satisfies $\bdl_{r,r+d+1,d}^{lex}\le \bdl_{r,r+d+1,d}^{revlex}$ by \eqref{eqn:upper-bound}.
    \item \label{step-2} 
        Next, we show that \eqref{eqn:sharp-bound-inter} holds when $c=2r+d-1$ via the $revlex$ type ordering. Whence, the right-hand side of \eqref{eqn:sharp-bound-inter} is given by $\bdl_{r,2r+d-1,d}^{revlex}$.
    \item \label{step-3}
        As for $c$ with $r+d+1\le c<2r+d-1$, we show that $\bdl_{r,c,d}^{revlex}+d\le \bdl_{r+1,c+1,d}^{revlex}$. Due to the format of the right-hand side of \eqref{eqn:sharp-bound-inter} and the established equalities in the previous two cases, we actually have equality here. Whence, the right-hand side of \eqref{eqn:sharp-bound-inter} is precisely $\bdl_{r,c,d}^{revlex}$ whenever $r+d<c<2r+d$.
    \item \label{step-4}
        In the last step, we construct a facet $F_0$ giving the desired upper bound $\bdl_{r,c,d}^{revlex}$, and show as in the previous section that the projective dimension in mind is exactly this bound.
\end{enumerate}

\subsubsection{\textbf{\textrm{Step}} \ref{step-1}}
In the extremal case when $c=r+d+1$, applying the condition in \eqref{condition:clique-end} with the $<_d$ requirement,
it is easy to check that all maximal cliques have the common starting variable
\[
    Y_{\bdbeta_1}=Y_{1,2+d,3+2d,...,r+(r-1)d},
\]
which is the first ring variable.
This implies that when calculating the colon ideal $I_F$, no tail generator will ever appear. Whence, the ideal $(\ini(P))^\vee$ has linear quotients, and the projective dimension of this ideal is precisely the maximal number of the corners by \cite[Corollary 8.2.2]{MR2724673}.

Now, we will apply the $lex$ type ordering. The computation of this subcase is based on the following key observation.

\begin{Remark}
    \label{rmk:lex-ordering}
    We have seen earlier in \Cref{rmk:moving-seq} that any maximal clique $F$ is determined by its moving sequence $\bdlambda_F$. We can break this sequence into maximally strictly increasing subsequences $\bdlambda_{1},\dots,\bdlambda_s$, just as what we did in \Cref{lambdaIncrease}.

    Suppose that when we are processing some subsequence $\bdlambda_t$ with $1\le t\le s$, we can move some variable at the position $p$ with $1\le p\le r$. But instead, in $F$, we move to other positions. This implies that this $p$ does not belong to $\bdlambda_t$, but rather, belongs to some $\bdlambda_{t'}$ with $t<t'$. Let the $t'$ be minimal with respect to this property. Then $p$ does not appear in any of $\bdlambda_{t},\dots,\bdlambda_{t'-1}$. One can check that if we move the $p$ to any of the $\bdlambda_t,\dots, \bdlambda_{t'-1}$, we still get an allowable moving sequence, i.e., the corresponding maximal clique is legitimate. Notice that after this change, the number of strictly increasing subsequences does not increase. And it decreases precisely when $\bdlambda_{t'}$ consists solely of $p$.

    The consequence of this observation is that among all maximal cliques with the same starting variable, a clique with the minimal number of strictly increasing subsequences can be obtained by applying the following strategy. Namely, after the previous round of movement, we scan the movement choices from $1$ to $r$ in order. For each choice $p$, if the corresponding movement is legitimate (it should satisfy the $<_d$ condition and the resulting variable should precede the final variable $Y_{\bdbeta_N}$), we adopt this choice into the moving sequence and consider the next choice $p+1$. 
    If the movement $p$ is not allowed, then we will consider the next choice $p+1$ directly.  After we have considered the last possible movement $r$, we call it a round and rewind to consider a new round. If the current $\bdbeta_t$ reaches the final one, we stop. 
\end{Remark}

Now, for the extremal case $c=r+d+1$, we only need to check with the unique  maximal clique with the aforementioned moving strategy. Regarding this clique, we can imagine the moving subsequence of the $i$-th round $\bdlambda_i$ as a \emph{virtual $r$-tuple} $\bdlambda_i=(\bdlambda_i^1,\dots,\bdlambda_i^r)$. It is not difficult to check that for each $j$ with $1\le j\le r$, we have
\[
    \bdlambda_i^j=
    \begin{cases}
        j,& \text{if $1+r\le i+j\le d+1+r$}, \\
        \text{void movement},&\text{otherwise}.
    \end{cases}
\]
In other words, the moving sequence is the concatenation of the rows of the $(d+r)\times r$ diagram in \Cref{fig:c=r+d+1}, starting from the top while void movements are left blank.
\begin{figure}[htb]
    \includegraphics[width=7.5cm]{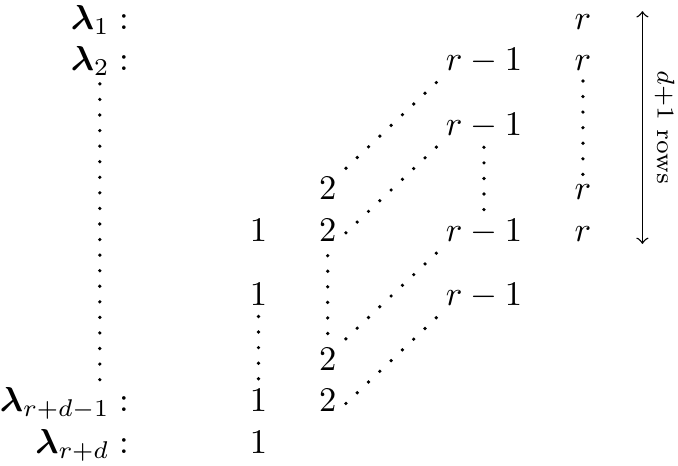} 
    \caption{Moving sequence in the case of $c=r+d+1$} 
    \label{fig:c=r+d+1}
\end{figure} 
    As a quick example, consider the case when $(r,c,d)=(3,5,1)$. The starting variable is $Y_{1,3,5}$ and consequently, the final variable has to be $Y_{3,5,7}$. The above moving strategy leads to the maximal clique
    \[
        \Set{Y_{1,3,5}\xrightarrow{3}
        Y_{1,3,6}\xrightarrow{2}
        Y_{1,4,6}\xrightarrow{3}
        Y_{1,4,7}\xrightarrow{1}
        Y_{2,4,7}\xrightarrow{2}
        Y_{2,5,7}\xrightarrow{1}
        Y_{3,5,7}}.
    \] 

Since we have $r+d$ rounds of movements and only the end movement at each round fails to contribute to any corner generator, the corresponding colon ideal contains exactly $r(d+1)-(r+d)=rd-d$ corner generators. In other words,
\[
    \pd((\ini(P))^\vee)=d(r-1),
\]
agreeing with the formula in \eqref{eqn:sharp-bound-inter}.

\subsubsection{\textbf{\textrm{Step}} \ref{step-2}}
In the extremal case when $c=2r+d-1$, we apply the $revlex$ type ordering of variables. Notice that
\[
    N-1=c+(r-1)d-1=r(2+d)-2.
\]
Furthermore, when $r=2$, $c=2r+d-1=r+d+1$, which has been covered by Step \ref{step-1}. Thus, we will assume that $r\ge 3$ in the following. It is clear that for each maximal clique $F$, the number of corners is bound from the above by 
\begin{equation}
    (N-1)-\ceil{(N-1)/r}=(r-1)(2+d)-2,
    \label{eqn:upper-bound-2r+d-1}
\end{equation}
which is precisely the right-side of \eqref{eqn:sharp-bound-inter} in this case.
This upper bound can be achieved by the following special maximal clique $F_0=\{Y_{\bdgamma_1}>\cdots >Y_{\bdgamma_N}\}$ with  the moving sequence as a concatenation in order of 
\begin{itemize}
    \item $1$ sequence of the form $(r,r-1,\dots,3,2)$, 
    \item $d$ sequences of the form $(r,r-1,\dots,2,1)$, and
    \item $1$ sequence of the form $(r-1,r-2,\dots,2,1)$. 
\end{itemize}
One can check that such a maximal clique exists and gives the desired number of corners via the description in \Cref{case:corner}. Furthermore,
\[
    \bdgamma_1=(1,2+d,2(2+d),\dots,(r-1)(2+d))
\]
and
\begin{equation}
    \bdgamma_N=(2+d,2(2+d),\dots,(r-1)(2+d),r(2+d)-1)
    \label{gamma-N}
\end{equation}
by \eqref{condition:clique-end}, \eqref{eqn:mv-determine-beta1-1}, and \eqref{eqn:mv-determine-beta1-2}.
As a quick example, when $(r,c,d)=(3,6,1)$, the maximal clique $F_0$ in mind will be
\[
    \Set{
        Y_{1,3,6}\xrightarrow{3}
        Y_{1,3,7}\xrightarrow{2}
        Y_{1,4,7}\xrightarrow{3}
        Y_{1,4,8}\xrightarrow{2}
        Y_{1,5,8}\xrightarrow{1}
        Y_{2,5,8}\xrightarrow{2}
        Y_{2,6,8}\xrightarrow{1}
        Y_{3,6,8}
    }.
\]

Regarding this clique, we claim that the tail generator of $F_0$ will be canceled out by its corners. Hence the minimal monomial generating set $G_{F_0}$ contains only corners.   
To see this, we notice that the last movement is $1<r$. Therefore $\delta_1=N$ in \eqref{eqn:delta}. Now, by \eqref{eqn:delta-k} and \eqref{gamma-N}, the special $k$ in \eqref{sequence-gamma} is at least $2$. Notice that $\delta_2=N-1$, which means that the tail generator contains the corner generator $Y_{\bdgamma_{N-1}}$.  
As a consequence of this claim, the minimal generating set $G(I_{F_0})$ consists of only corner generators. Whence,
\begin{equation}
    \Ess(F_0)=\{
    Y_{\bdgamma_1},Y_{\bdgamma_{r}},Y_{\bdgamma_{2r}},\dots,Y_{\bdgamma_{(d+1)r}},Y_{\bdgamma_{N}}
    \}.
    \label{eqn:Ess-F_0}
\end{equation}

It remains to show that the integer in \eqref{eqn:upper-bound-2r+d-1} gives the desired projective dimension, i.e., to verify the conditions \ref{C1}-\ref{C3} in \Cref{F_0}. The \ref{C1} part is automatic.


As for the condition \ref{C2}, 
suppose that $F=\Set{Y_{\bdbeta_1}>\cdots>Y_{\bdbeta_N}}$ is a maximal clique with $\Ess(F)=\Ess(F_0)$.

We have seen in \Cref{rmk:F_0} that $Y_{\bdbeta_1}=Y_{\bdgamma_1}$. 
Consequently, $Y_{\bdbeta_N}=Y_{\bdgamma_N}\in \Ess(F_0)=\Ess(F)$. This implies that $G(I_F)$ consists of only corner generators.
Whence, the moving subsequences connecting each adjacent pair in $\Ess(F)$ regarding $F$ are all strictly decreasing. Notice that each such strictly decreasing moving subsequence is completely determined by its two terminal variables in $F$.
Since $\Ess(F)=\Ess(F_0)$, the moving sequence of $F$ agrees with that of $F_0$. This is equivalent to saying that $F=F_0$, verifying the condition \ref{C2}. 

It remains to verify \ref{C3}. Indeed, since other subcases are similar as in  \Cref{subsection:c>=2r+d}, we only need to consider the case when the maximal clique $F'=\{Y_{\bdalpha_1}>\cdots >Y_{\bdalpha_N}\}$ satisfies $\bdalpha_1=\bdgamma_1$ (consequently $\bdalpha_N=\bdgamma_N$) and $\Ess(F')=\Ess(F_0)\setminus \{Y_{\bdalpha_N}\}$. 
Now, the final $r-1$ movement of $F'$ is just a rearrangement of $(r-1,r-2,\dots,2,1)$, not containing $r$. 
Furthermore, since $\bdgamma_1^3=2(2+d)$,
it follows from \eqref{eqn:1-k-minimal} that the special $k$ in \eqref{sequence-gamma} for $F'$ is at most $2$. 
Notice that $\delta_2\ge N-1$ for $F'$. In turn, the degree of the tail generator for $F'$ is at most $N-(N-1)+1=2$, and it equals $2$ exactly when $\delta_2=N-1$ for $F'$, or equivalently, the last movement is $1$. 
\begin{enumerate}[i]
    \item When the last movement is $1$, the tail generator is quadratic and all the remaining movements (a rearrangement of $(r-1,r-2,\dots,3,2)$) contribute to corners by expression of $\Ess(F')$. Whence, the corresponding moving subsequence is strictly decreasing and has to be exactly $(r-1,r-2,\dots,3,2)$. This means that $\bdlambda_{F'}=\bdlambda_{F_0}$, and in turn $F'=F_0$, a contradiction. 
    \item When the last movement is not $1$, the tail generator is linear. However, if the last movement of position $1$ in the moving sequence is the movement from $Y_{\bdalpha_{j}}$ to $Y_{\bdalpha_{j+1}}$ with $N-r+1\le j<N-1$, it is the final movement of some maximally strictly decreasing subsequence. Whence, $Y_{\bdalpha_{j+1}}\in \Ess(F')\setminus \Ess(F)$, another contradiction. And this completes our argument for the case when $c=2r+d-1$.
\end{enumerate}

\subsubsection{\textbf{\textrm{Step}} \ref{step-3}}
In this subsection, we focus on the case when $r+d+1\le c<2r+d-1$ and show that $\bdl_{r,c,d}^{revlex}+d\le \bdl_{r+1,c+1,d}^{revlex}$. Whence, $r\ge 3$. 
Obviously, we will apply the $revlex$ type ordering of variables in the following.
Let $F$ be a maximal clique in the case $(r,c,d)$ such that $\ell_F=\bdl_{r,c,d}^{revlex}$. Let $\bdlambda_F$ be the corresponding moving sequence and break it into maximally strictly decreasing subsequences $\bdlambda_{1},\bdlambda_{2},\dots,\bdlambda_s$. Since the length of each subsequence is at most $r$, it is clear that  
\[
    s\ge \ceil{\frac{N-1}{r}}\ge \ceil{\frac{r+d+1+(r-1)d-1}{r}}=d+1
\] 
for $N=N(r,c,d)=c+(r-1)d$.
Suppose that $F=\Set{Y_{\bdbeta_1}>\cdots>Y_{\bdbeta_N}}$. We will construct a related maximal clique $\widetilde{F}=\Set{\widetilde{Y}_{\bdalpha_1}>\cdots>\widetilde{Y}_{\bdalpha_{\widetilde{N}}}}$ in the case $(r+1,c+1,d)$ as follows. 
As a reminder, the ring for $\widetilde{F}$ will be $\KK[\widetilde{\bdY}]=\KK[\widetilde{Y}_{\bdalpha}\,|\,\bdalpha\in \Lambda_{r+1,d}(\widetilde{N})]$ where $\widetilde{N}=N(r+1,c+1,d)=c+1+(r+1-1)d=N+d+1$.
The $(r+1)$-tuple ${\bdalpha_1}$ will be obtained by appending $\bdbeta_1$ with $N=c+(r-1)d$. The facet $\widetilde{F}$ is then determined by the moving sequence $(\widetilde{\bdlambda_0},\widetilde{\bdlambda_{1}},\dots,\widetilde{\bdlambda_s})$, where
\begin{itemize}
    \item $\widetilde{\bdlambda_0}=r+1$ is a sequence containing just one element,
    \item $\widetilde{\bdlambda_i}$ is obtained by prepending $\bdlambda_i$ by $r+1$ for $1\le i\le d$, and
    \item $\widetilde{\bdlambda_i}=\bdlambda_i$ for $d+1\le i\le s$.
\end{itemize}
As a quick example, consider the case $(r,c,d)=(3,5,1)$. Then  
\[
F=\Set{
Y_{1,3,5}>Y_{1,3,6}>Y_{1,3,7}>Y_{1,4,7}>Y_{1,5,7}>Y_{2,5,7}>Y_{3,5,7}
}
\]
with the moving sequence $\bdlambda_F=(\bdlambda_1,\bdlambda_2,\bdlambda_3,\bdlambda_4)=(\underline{3}, \underline{3, 2}, \underline{2, 1}, \underline{1})$ is a maximal clique with $\ell_F=\bdl_{3,5,1}^{revlex}$. The newly constructed maximal clique will be
\[
\widetilde{F}=\Set{
\widetilde{Y}_{1,3,5,7}>\widetilde{Y}_{1,3,5,8}>\widetilde{Y}_{1,3,5,9}>\widetilde{Y}_{1,3,6,9}>\widetilde{Y}_{1,3,7,9}>\widetilde{Y}_{1,4,7,9}>\widetilde{Y}_{1,5,7,9}>\widetilde{Y}_{2,5,7,9}>\widetilde{Y}_{3,5,7,9}
}
\]
with the moving sequence $\bdlambda_{\widetilde{F}}=(\widetilde{\bdlambda_0},\widetilde{\bdlambda_1},\widetilde{\bdlambda_2},\widetilde{\bdlambda_3},\widetilde{\bdlambda_4})=
(\underline{4}, \underline{4, 3}, \underline{3, 2}, \underline{2, 1}, \underline{1})$.

Now, back to the construction. We have the following observations.
\begin{enumerate}[i]
    \item Using \Cref{MaxCli} and \Cref{rmk:moving-seq}, we can verify that the above moving sequence $(\widetilde{\bdlambda_0},\widetilde{\bdlambda_{1}},\dots,\widetilde{\bdlambda_s})$ is allowable. Equivalently, we get a legal maximal clique in the case of $(r+1,c+1,d)$.
    \item \label{step-3-fact-2}
        Notice that the final subsequences coincide: $\widetilde{\bdlambda_i}=\bdlambda_i$ for $d+1\le i\le s$. It follows that
        $\bdalpha_{\widetilde{N}-i}^{\le r}=\bdbeta_{N-i}$ 
        and $\bdalpha_{\widetilde{N}-i}^{r+1}=\widetilde{N}$
        for $i\le \sum_{j=d}^s |\bdlambda_j|$. Here, $|\bdlambda_j|$ is the length of the sequence $\bdlambda_j$.
    \item \label{step-3-fact-3}
    For the positions in the moving sequence of $F$ that induce corner generators, the corresponding positions in the moving sequence of $\widetilde{F}$ still induce corner generators.
    Meanwhile, the initial movement $r+1$ of $\widetilde{\bdlambda_i}$ for $1\le i\le d$ will contribute additional corner generators.
    And then, we have the complete list of corner generators for $\widetilde{F}$.
    \item Furthermore, using the formula established in Step \ref{step-2} as the base argument, by induction on applying the expected inequality $\bdl_{r,c,d}^{revlex}+d\le \bdl_{r+1,c+1,d}^{revlex}$, we will have 
    \begin{align}
        \ell_F&=\bdl_{r,c,d}^{revlex}\ge 
        (r-(2r+d-1-c)-1)(2+d)-2+(2r+d-1-c)d \notag \\
        &=dr-2r-3d+2c-2.
        \label{ineq-1}
    \end{align}
\end{enumerate}

Now, let $\tau_F$ and $\tau_{\widetilde{F}}$ be the tail generators for the maximal clique $F$ and the newly constructed $\widetilde{F}$ respectively.
If $\tau_F\notin G_F$, then no matter whether $\tau_{\widetilde{F}}\in G_{\widetilde{F}}$, we have already $\ell_F+d\le \ell_{\widetilde{F}}$ by the above item \ref{step-3-fact-3}.
Thus, we will assume that $\tau_F\in G_F$ in the following and intend to prove that $\tau_{\widetilde{F}}\in G_{\widetilde{F}}$.
For that purpose, we first claim that 
\begin{equation}
    \deg(\tau_F) \le -1+\sum_{j=d}^s |\bdlambda_j|.
    \label{claim-tail}
\end{equation} 
Otherwise, all corner generators of $F$ will be generated from the first $d-1$ moving sequences. Note that each such subsequence has length at most $r$, and consequently contributes at most $r-1$ corner generators. Whence, with the additional contribution from the tail generator, we have
\[
    \ell_F\le (d-1)(r-1)+1.
\]
But this contradicts the inequality in \eqref{ineq-1} since $c\ge r+d+1$ and $r\ge 3$.

Suppose in addition that the tail generator $\tau_F$ is given by $\bdY^H$ with $H=\Set{Y_{\bdbeta_{\delta_k}}>\cdots >Y_{\bdbeta_N}}$ at the end of \Cref{case:tail} for some $k\le r-1$. Recall that $\delta_k$ was defined in \eqref{eqn:delta}. We can similarly introduce $\widetilde{\delta}_k$ for $\widetilde{F}$ with respect to this $k$ as
\[
    \widetilde{\delta}_k\coloneqq \min\Set{i| \bdalpha_i^{\{k,...,r\}}=\bdalpha_{\widetilde{N}}^{\{k,...,r\}}}.
\]
It follows from the above item \ref{step-3-fact-2} and the inequality \eqref{claim-tail} that $N-\delta_k\ge \widetilde{N}-\widetilde{\delta}_{k}$.
Notice that
$
\bdalpha_{\widetilde{\delta}_{k}-1}^k= \bdalpha_{\widetilde{\delta}_{k}}^k
$ or $\bdalpha_{\widetilde{\delta}_{k}}^k-1$ while
$
\bdbeta_{\delta_k-1}^k=\bdbeta_{\delta_k}^k
$ or $\bdbeta_{\delta_k}^k-1$ by the condition \eqref{condition:move-1}. 
Meanwhile,
\[
    \bdalpha_{\widetilde{\delta}_{k}}^k=\bdbeta_{N-(\widetilde{N}-\widetilde{\delta}_{k})}^k\ge\bdbeta_{\delta_k}^k
\]
by the above item \ref{step-3-fact-2} and the inequalities \eqref{eqn:clique-condition} and \eqref{claim-tail}. Since
$1+(k-1)(d+1) <_d \bdbeta_{\delta_k-1}^k$ by the inequality \eqref{eqn:delta-k}, we will also have $1+(k-1)(d+1) <_d \bdalpha_{\widetilde{\delta}_{k}-1}^k$, unless 
\[
    \bdalpha_{\widetilde{\delta}_{k}-1}^k+1= \bdalpha_{\widetilde{\delta}_{k}}^k
    =\bdbeta_{\delta_k-1}^k=\bdbeta_{\delta_k}^k.
\]
But by the redundancy $-1$ in the inequality \eqref{claim-tail} as well as the fact that
$N-\delta_k\ge \widetilde{N}-\widetilde{\delta}_{k}$ and the above item \ref{step-3-fact-2}, we will have instead
\[
    \bdalpha_{\widetilde{\delta}_k-1}^k=\bdbeta_{N-\widetilde{N}+\widetilde{\delta}_k-1}^k\ge \bdbeta_{\delta_k-1}^k, 
\]
a contradiction.

Now as $1+(k-1)(d+1) <_d \bdalpha_{\widetilde{\delta}_{k}-1}^k$, by the description of the tail generators in \Cref{case:tail}, $\tau_{\widetilde{F}}=\widetilde{Y}^{\widetilde{H}}$ for some $\widetilde{H}=\Set{\widetilde{Y}_{\bdalpha_{\widetilde{\delta}_{k'}}},\dots,\widetilde{Y}_{\bdalpha_{\widetilde{N}}}}$ where $k'\le k$.
Consequently, $\deg(\tau_F)\ge \deg(\tau_{\widetilde{F}})$. 
As $\tau_F$ belongs to the minimal generating set $G_F$ of the colon ideal of $F$, $\tau_F$ won't be canceled out by the corner generators of $F$. By our previous description of the corner generators for $\widetilde{F}$ in item \ref{step-3-fact-3}, $\tau_{\widetilde{F}}$ won't be canceled out by the corner generators of $\widetilde{F}$. Whence, $\tau_{\widetilde{F}}\in G_{\widetilde{F}}$. 
Since in this case, both $F$ and $\widetilde{F}$ have tail generators, we have again shown that $\ell_{F}+d\le \ell_{F'}$ in view of the previous item \ref{step-3-fact-3}. 

Consequently, $\bdl_{r,c,d}^{revlex}+d\le \bdl_{r+1,c+1,d}^{revlex}$, as expected.  

\subsubsection{\textbf{Step} \ref{step-4}}
In this final step, we will construct a maximal clique $F_0$ giving the desired upper bound $\bdl_{r,c,d}^{revlex}$, and show as in the previous section that the projective dimension in mind is exactly this bound. 
As the extremal cases when $c=r+d+1$ and $c=2r+d-1$ have already been shown previously, we will assume that $r+d+1<c<2r+d-1$ here. Whence, $r\ge 4$. 

Write temporarily that $\varepsilon=2r+d-1-c$, $c_0=c-\varepsilon$ and $r_0=r-\varepsilon$. Notice that $c_0=2r_0+d-1$. And as $1\le \varepsilon \le r-3$, $r_0\ge 3$. We start with the special ``maximal'' moving sequence given in Step \ref{step-2}. Then we can apply the argument in Step \ref{step-3} successively to build a ``maximal'' moving sequence as the concatenation of the rows of the $(d+2+\varepsilon)\times r$ diagram in \Cref{fig:intermediate-c}, starting from the top.
\begin{figure}[htb]
    \includegraphics[width=13cm]{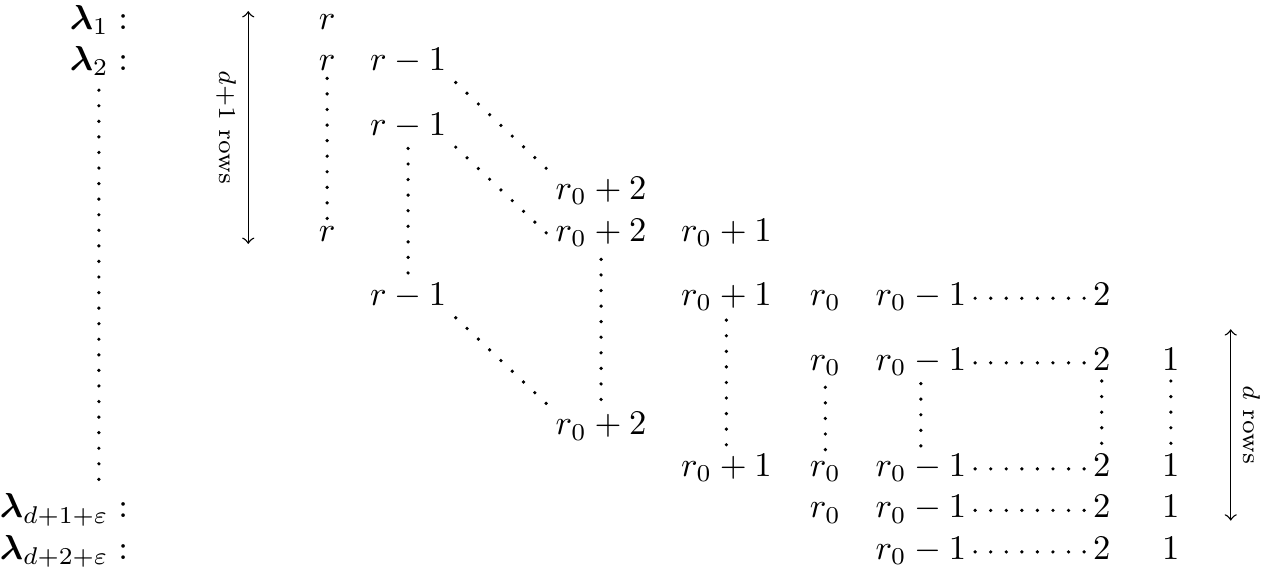}  
    \caption{Moving sequence in the case of $r+d+1<c<2r+d-1$}
    \label{fig:intermediate-c}
\end{figure}

Let $F_0$ be the corresponding maximal clique. Since the special starting maximal clique considered in Step \ref{step-2} has no tail generator in the corresponding minimal generating set, so does $F_0$ here. Otherwise, we would have strict inequality in some intermediate processes when applying the constructions in Step \ref{step-3}. Thus,
\[
    \ell_F > \bdl_{r_0,c_0,d}^{revlex}+ \varepsilon d =\bdl_{r,c,d}^{revlex},
\]
a contradiction.

Furthermore, by the descriptions of the corner generators in Steps \ref{step-2} and \ref{step-3}, it is clear that now in every maximally decreasing moving subsequence of $F_0$, only the final movement fails to contribute a corner generator.

To finish the proof, it remains to verify the corresponding conditions \ref{C1}-\ref{C3} in \Cref{F_0}. 
The condition \ref{C1} is automatic. As for the condition \ref{C2}, like the argument for Step \ref{step-2}, we notice that the first variable and the last variable of $F_0$ belong to $\Ess(F_0)$. Thus, if the maximal clique $F'$ satisfies $\Ess(F')=\Ess(F_0)$, then $F'$ has the same starting variable, final variable, and same corner generators as $F_0$. Consequently, $F'$ has the same moving sequence as $F_0$, making $F'=F_0$.
As for the condition \ref{C3}, again, we will only need to consider the case when $F'$ has the same starting variable as $F_0$, by \Cref{rmk:F_0}. 
\begin{enumerate}[i]
    \item Suppose that $\Ess(F')=\Ess(F_0)\setminus \{Y_{\bdbeta_N}\}$ in \ref{C3}. 
        Now, the final $r_0-1$ movements of $F'$ is just a rearrangement of $(r_0-1,r_0-2,\dots,2,1)$. And the remaining argument will be similar to that at the end of Step \ref{step-2}.
    \item Otherwise, $f$ in \ref{C3} has to be the position after some $k$-th round of movements $\bdlambda_k$ of $F_0$ for some $k<d+2+\varepsilon$. Let $f'$ and $f''$ be the positions in $\Ess(F_0)$ immediately before and after $f$ respectively. Then $f',f''\in \Ess(F')$ as well. Consequently, all the positions between them in $F'$ will contribute corner generators for $F'$, and the corresponding movements will form a strictly decreasing subsequence $\bdlambda'$ of $F'$, which is a rearrangement of the concatenation $\bdlambda_k, \bdlambda_{k+1}$. 
        Let $m$ be the first movement in $\bdlambda_{k+1}$ of $F_0$. It is clear from \Cref{fig:intermediate-c} that
        $m$ will appear twice in $\bdlambda'$, since $r_0\ge 3$. This is a contradiction.
\end{enumerate}

And this completes our construction for Step \ref{step-4}, which in turn finished the proof of \Cref{prop:sharp-bound-inter}.

\subsection{The $c\le r+d$ case}
\label{sec:less-than-r+d}
Here, we consider the regularity in the degenerated case when $r\le c\le r+d$.  We start with the extremal case when $c=r+d$.

\begin{Lemma}
    \label{lem:c=r+d}
    Suppose that $c= r+d$. Then the projective dimension is given by
    \begin{equation*}
        \projdim((\ini(P))^\vee)=(r-1)(d-1). 
    \end{equation*}
\end{Lemma}

\begin{proof}
    It is time to go over the key parts stated in the first three subsections earlier.
    Recall that the dimension of $\calF(I)$ is then $rc-r^2+1=rd+1$ instead of $N=c+(r-1)d$ by \Cref{Deformation}. Now,
    the minimal monomial generating set of $\ini(P)^{\vee}$ is still given by \eqref{eqn:max-clique-gen}. However, all the maximal cliques take the form
    \begin{equation*}
        F=\Set{Y_{\bdbeta_1}>\cdots>Y_{\bdbeta_{rd+1}}},
    \end{equation*}
    where
    \[
        \bdbeta_1^j=j+(j-1)d \qquad \text{for $j=1,2,\dots,r$,}
    \]
    and
    \[
        \bdbeta_{rd+1}^j=N-(r-j)(d+1) \qquad \text{for $j=1,2,\dots,r$.}
    \]
    Here, the index $\bdbeta_1$ corresponds to the main diagonal of the leftmost maximal minor of $\bdH_{r,c,d}$, and $\bdbeta_{rd+1}$ corresponds to the rightmost one. Surely we don't have \eqref{condition:clique-end}. As for \eqref{eqn:clique-condition} and \eqref{condition:move-1}, we only need to change the corresponding $N$ into $rd+1$. 

    We will apply the $lex$ type ordering to the minimal generating set $G( (\ini(P))^{\vee})$.
    The description will be completely the same as in \Cref{ss:linear-quotients}.
    Since all maximal cliques have common starting and ending variables, there is no tail generator in any of the colon ideals. The description of the corners is identical to that stated in \Cref{case:corner}. In particular, $(\ini(P))^{\vee}$ has linear quotients, and its projective dimension is achieved exactly by the maximal cardinality of the corners sets.
    Undoubtedly, this maximal length is achieved by applying the strategy in \Cref{rmk:lex-ordering}.

    Due to the explicit description of the starting variable $Y_{\bdbeta_1}$, the moving sequence is then the concatenation of the rows of the $(d+r-1)\times r$ diagram in \Cref{fig:c=r+d}, starting from the top.
    \begin{figure}[htb]
        \includegraphics[width=7.5cm]{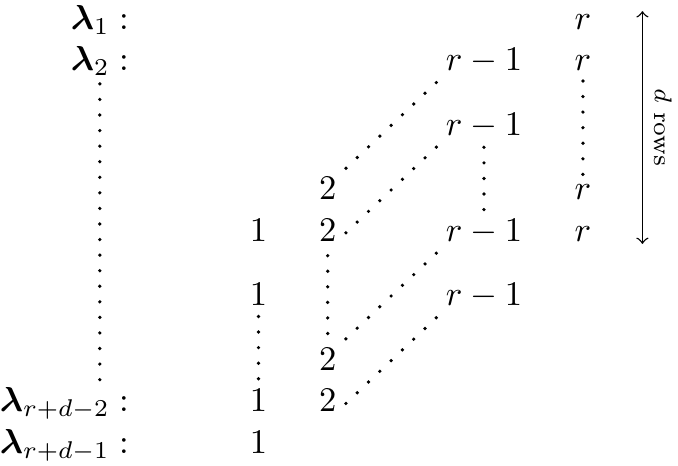} 
        \caption{Moving sequence in the case of $c=r+d$} 
        \label{fig:c=r+d}
    \end{figure}
    It is not difficult to see that it contains  $rd-(r+d-1)=(r-1)(d-1)$ corners. In other words, the expected projective dimension is given by $(r-1)(d-1)$.
\end{proof}

Now, we consider the general case.

\begin{Proposition}
    \label{prop:less-than-r+d}
    Suppose that $r\le c\le r+d$. Then the projective dimension is given by
    \begin{equation*}
        \projdim((\ini(P))^\vee)=
        \begin{cases}
            (r-1)(c-r-1), & \text{if $r<c$},\\
            0, & \text{if $r=c$}.
        \end{cases}
    \end{equation*}
\end{Proposition}

\begin{proof}
    When $r=c$, $I_r(\bdH_{r,c,d})$ is principal of degree $r$.
    And when $c=r+1$, one can check that all pairs of the elements in $\Lambda_{r,d}(N)$ are comparable.
    In both cases, the defining ideal $P=0$ and the regularity of the fiber cone is $0$. 

    When $1<r<r+1<c\le r+d$, by the reduction before \cite[Theorem 2.3]{arXiv:1901.01561}, we can reduce the ideal $I_r(\bdH_{r,c,d})\subset R=\KK[\bdx]$ to some $I_r(\bdH_{r,c,d'})\subset \KK[{\bdx'}]$ with $d'=c-r$.  Here, the collection of variables ${\bdx'}$ is a subset of the original collection of variables $\bdx$, and $I_r(\bdH_{r,c,d'})\KK[\bdx]=I_r(\bdH_{r,c,d})$. Whence, $c=r+d'$ and by \Cref{lem:c=r+d}, the projective dimension is 
    \[
        \projdim((\ini(P'))^\vee)=(r-1)(d'-1),
    \]
    where $P'$ is the defining ideal of $\calF(I_r(\bdH_{r,c,d'}))$ as $P$ for $\calF(I_r(\bdH_{r,c,d}))$.
    Therefore, the original projective dimension is
    \begin{equation*}
        \projdim((\ini(P))^\vee)=(r-1)(c-r-1). 
    \end{equation*}
    when $1<r<c\le c+d$. 
\end{proof}

Therefore, we have completed the proof for our \Cref{thm:main-reg-balanced}.

\begin{Remark}
    Knowing the extremal Betti numbers of $\ini(P)$ amounts to knowing the extremal Betti numbers of $(\ini(P))^{\vee}$, by \cite[Theorem 5.61]{MR2110098}. Since $(\ini(P))^{\vee}$ is known to have linear quotients, the Cohen--Macaulay type of the fiber cone $\calF(\ini(I))$ is simply the top Betti number of $(\ini(P))^{\vee}$. And our approach in this section then paves a road towards handling it. In particular, one can try to characterize when $\calF(\ini(I))$ is Gorenstein. But this has already been done neatly in \cite[Theorem 3.7]{arXiv:1901.01561}, which shows that $\calF(\ini(I))$ is Gorenstein if and only if
    \[
        c\in \Set{r,r+1,r+d,r+d+1,2r+d}
    \]
    when $r\ge 2$ and $d\ge 1$. The paper \cite{arXiv:1901.01561} takes advantage of the Ehrhart ring theory, which is a standard combinatorial tool for handling this type of problem. As a quick corollary, since $\calF(I)$ and $\calF(\ini(I))$ have the same Cohen--Macaulay type, the original fiber cone $\calF(I)$ is Gorenstein if and only if the number of columns $c$ satisfies the same requirement, as mentioned earlier in \Cref{Deformation}.

    It is also worth mentioning that the paper \cite{arXiv:1901.01561} computed some geometric invariant $\delta$ of the associated integral convex polytope. When the corresponding fiber cone $\calF(\ini(I))$ is Gorenstein, this integer $\delta$ is simply $-\bda(\calF(\ini(I)))$ by \cite[Proposition 2.2]{Noma}, where $\bda(\calF(\ini(I)))$ is the \emph{$\bda$-invariant} introduced by Goto and Watanabe in \cite[Definition 3.1.4]{MR494707}. Whence, we can derive the corresponding regularity for free. This is because 
    \begin{equation}
        \bda(A)=\reg(A)-\dim(A)
        \label{def:a-inv}
    \end{equation}
    for any standard graded Cohen--Macaulay algebra $A$ over $\KK$, in view of the equivalent definition of regularity in \cite[Definitions 1 and 3]{MR676563}.
    It is not surprising that the outcome agrees with our formula in \Cref{thm:main-reg-balanced} for these Gorenstein cases.
\end{Remark}

We end this section with a quick application. It is also due to the following fact.

\begin{Lemma}
    [{\cite[Proposition 6.6]{CNPY} or \cite[Proposition 1.2]{MR3864202}}]
    \label{CNPY:6.6}
    Let $I \subset R = \KK[x_1,\dots,x_N]$ be a homogeneous ideal that is generated in one degree, say $d$. Assume that the fiber cone $\calF(I)$ is Cohen--Macaulay. Then each minimal reduction of $I$ is generated by $\dim(\calF(I))$ homogeneous polynomials of degree $d$, and $I$ has the reduction number $\r(I) = \reg(\calF(I))$.
\end{Lemma}

\begin{Corollary} 
    \label{reduction}
    The reduction numbers of the ideal $I_r(\bdH_{r,c,d})$ and its initial ideal $\ini(I_r(\bdH_{r,c,d}))$ are given by
    \[
        \r(I_r(\bdH_{r,c,d}))=
        \r(\ini(I_r(\bdH_{r,c,d})))=
        \begin{cases}
            N-1-\floor{(N-1)/r}, & \text{if $2r+d\le c$},\\
            dr-2r-3d+2c-2, & \text{if $r+d< c< 2r+d$},\\
            (r-1)(c-r-1), & \text{if $r< c\le r+d$},\\
            0, & \text{if $r=c$},
        \end{cases}
    \]
    where $N=c+(r-1)d$. And the $\bda$-invariants of $\mathcal{F}(I_r(\bdH_{r,c,d}))$ and  $\mathcal{F}(\ini (I_r(\bdH_{r,c,d})))$ are given by  
    \[
        \bda(\mathcal{F}(I_r(\bdH_{r,c,d})))=
         \bda(\mathcal{F}(\ini I_r(\bdH_{r,c,d})))=
        \begin{cases}
            -1-\floor{(N-1)/r}, & \text{if $2r+d\le c$},\\
            c-2r-2d-2, & \text{if $r+d< c< 2r+d$},\\
            -c, & \text{if $r< c\le r+d$},\\
            -1, & \text{if $r=c$}.
        \end{cases}
    \]
\end{Corollary} 

\begin{proof}
    By \Cref{Fiber-Sagbi}, \Cref{Deformation} \ref{Deformation-a} and  \cite[Corollary 3.4]{MR4019342}, the initial algebra of the fiber cone $\mathcal{F}(I_r(\bdH_{r,c,d}))$ is the fiber cone $\mathcal{F}(\ini (I_r(\bdH_{r,c,d})))$, and these two algebras are both Cohen--Macaulay. Now, it suffices to apply \cite[Corollary 2.5]{Sagbi}, \Cref{Deformation} \ref{Deformation-c}, \Cref{thm:main-reg-balanced}, \Cref{CNPY:6.6} and Equation \eqref{def:a-inv}.
\end{proof}

\begin{Remark}
    The initial ideal $\ini(I_r(\bdH_{r,c,d}))$ is also considered as the $(d+1)$-spread Veronese ideal of degree $r$ in \cite{MR4019342}. This concept was later generalized to the class of $c$-bounded $t$-spread Veronese ideals $I_{c,(n,d,t)}$ and the class of Veronese ideals of bounded support $I_{(n,d,t),k}$ in \cite{arXiv:2005.09601}. The regularity of the particular fiber cone $\KK[I_{(n,d,0),k}]$ was computed in \cite[Proposition 5.6]{arXiv:2005.09601}, which has a similar flavor as that in our \Cref{thm-main-result}. It is then natural to ask for the regularity of the fiber cone of other ideals considered in \cite{arXiv:2005.09601}.
\end{Remark}

\begin{acknowledgment*}
    The authors are grateful to the software system \texttt{Macaulay2} \cite{M2}, for serving as an excellent source of inspiration.
    The second author is partially supported by the ``Anhui Initiative in Quantum Information Technologies'' (No.~AHY150200) and the ``Fundamental Research Funds for the Central Universities''.
\end{acknowledgment*}

\bibliography{Regularity}

\end{document}